\documentclass[11pt, letter, oneside]{amsart}
    \usepackage{algpseudocode, algorithm}
    \usepackage[symbol]{footmisc}
    \usepackage{indentfirst}
    \usepackage{afterpage}
    \usepackage{array, makecell, float, longtable}
    \usepackage{amsmath}
    \usepackage{amssymb}
    \usepackage{amsthm}
    \usepackage{subcaption}
    \usepackage[utf8]{inputenc}
    \usepackage[margin=1in]{geometry}
    \usepackage{diagbox}
	\usepackage{enumerate}
    \usepackage{adjustbox}
    \usepackage{siunitx}
    \usepackage{graphicx}
    \usepackage{bbm}
    \usepackage{multirow}
    \usepackage{xcolor}
    \usepackage{tikz-cd}
    \usetikzlibrary{shapes.geometric}
    \usetikzlibrary{arrows.meta}
    \usepackage{dynkin-diagrams}
    \usepackage{mathrsfs}
    \usepackage{bm}
    \usepackage{epic}
    \usepackage[colorlinks = true,
            linkcolor = black,
            urlcolor  = black,
            citecolor = blue,
            hypertexnames = false]{hyperref}
    \usepackage[backend=bibtex, firstinits=true, sortcites=true, style=alphabetic]{biblatex}
    \usepackage[capitalize, noabbrev, nameinlink]{cleveref}

	\newtheorem{theorem}{Theorem}[section]
    \newtheorem{theoremdefinition}[theorem]{Theorem-Definition}
    \newtheorem{lemma}[theorem]{Lemma}
    \newtheorem{assumption}[theorem]{Assumption}
    \newtheorem*{assumption*}{Assumption}

	\newtheorem{corollary}[theorem]{Corollary}
	\newtheorem{proposition}[theorem]{Proposition}
	\theoremstyle{definition}
	\newtheorem{remark}[theorem]{Remark}

    \AddToHook{env/proposition/begin}{\crefalias{theorem}{proposition}}
    \AddToHook{env/lemma/begin}{\crefalias{theorem}{lemma}}
    \AddToHook{env/corollary/begin}{\crefalias{theorem}{corollary}}
    \AddToHook{env/conjecture/begin}{\crefalias{theorem}{conjecture}}
    \AddToHook{env/question/begin}{\crefalias{theorem}{question}}
    \AddToHook{env/remark/begin}{\crefalias{theorem}{remark}}
    \AddToHook{env/example/begin}{\crefalias{theorem}{example}}
    \AddToHook{env/fact/begin}{\crefalias{theorem}{fact}}
    \AddToHook{env/definition/begin}{\crefalias{theorem}{definition}}
    \AddToHook{env/assumption/begin}{\crefalias{theorem}{assumption}}

    \newcommand{\s}[0]{\sigma}
    \newcommand{\Z}[0]{\mathbb{Z}}
    \newcommand{\R}[0]{\mathbb{R}}

    \newcommand{\Q}[0]{\mathbb{Q}}

    \newcommand{\comment}[1]{}

    \newcommand{\tr}{\operatorname{tr}}

    \renewcommand{\H}{\mathbb H}
    \newcommand{\cd}{\operatorname{cd}}
    \newcommand{\vol}{\operatorname{vol}}

    \author{Lizi Guo}
    \address{Mathematics Department, University of Wisconsin–Madison, 480 Lincoln Dr, Madison, WI 53706, USA}
    \email{llguo@wisc.edu}

  \author{Jiming Ma}
    \address{School of Mathematical Sciences, Fudan University, Shanghai, 200433, P. R. China}
    \email{majiming@fudan.edu.cn}

    \author{Yourong Zang}
    \address{Mathematics Department, University of Wisconsin–Madison, 480 Lincoln Dr, Madison, WI 53706, USA}
    \email{yzang27@wisc.edu}

    \author{Fangting Zheng}
    \address{Department of Mathematical Sciences, Xi'an Jiaotong-Liverpool University, Suzhou 215123,	China}
    \email{fangting.zheng@xjtlu.edu.cn}

    \keywords{Coxeter polytopes, hyperbolic orbifolds, commensurability, exponential growth}

    \subjclass[2010]{52B11, 51F15, 51M10}

     \date{July 16, 2026}

    \thanks{Jiming Ma was supported by  NSFC No. 12171092. Fangting Zheng was supported by NSFC No. 12471067.}

    \title[Families of incommensurable polytopes]{Several families of incommensurable noncompact hyperbolic Coxeter polytopes}
    
\begin{document}
    \begin{abstract}
        We classify all 141 finite-volume hyperbolic Coxeter five-dimensional polytopes with eight facets, of which 125 are noncompact. Using maximal-cusp density and a noncompact analog of Bogachev–Douba–Raimbault's argument, we construct infinitely many pairwise incommensurable noncompact Coxeter polytopes in dimensions 4, 5, 6, 7, and 9, with the number of commensurability classes growing at least exponentially in volume.
    \end{abstract}
    \maketitle
    \section*{Introduction}
    The classification of hyperbolic Coxeter polytopes in dimension $\geq 3$, namely a convex polytope in $\H^n$ for $n\geq 3$ bounded by hyperplanes meeting at dihedral angles of $\frac{\pi}{k}$ for integers $k\geq 2$, remains undone. The first question of existence is already far from complete. Vinberg showed in 1985 \cite{vinberg-1985} that no compact hyperbolic Coxeter polytopes exist in dimension $\geq 30$. A year later, Prokhorov \cite{prokhorov-1986} proved the nonexistence of noncompact Coxeter polytopes of finite volume in dimension $\geq 996$. Instances of the former were exhibited only up to dimension 8 \cite{bugaenko-compact-1,bugaenko-compact-2}, and only up to dimension 21 for the latter \cite{vinberg-noncompact,vk-noncompact,borcherds-noncompact}. Allcock in \cite{allcock-infinite} constructed infinitely many compact and noncompact Coxeter polytopes of finite volume in $\mathbb{H}^n$ for $n\leq 6$ and $n\leq 19$ respectively. Even in small dimensions, no complete classification is known. To push the current status a bit further, the first part of this paper is dedicated to the complete classification of all five-dimensional Coxeter polytopes of finite volume with eight facets. 
    \begin{theorem}\label{thm:58list}
        	There are exactly 141 finite-volume hyperbolic Coxeter 5-polytopes with 8 facets, out of which 16 are compact, and 125 are noncompact.
    \end{theorem}
    Allcock's original constructive argument ceases to work when it is passed to commensurability classes of Coxeter polytopes. Two $n$-dimensional Coxeter polytopes $P_1, P_2$ are said to be \textit{commensurable} whenever they admit a finite-sheeted common cover. In the context of compact hyperbolic manifolds, Raimbault in \cite{raimbault-manifold-comm} produced a lower bound for the number of commensurability classes of bounded volumes that grows exponentially. This growth rate was later improved and generalized by Gelander and Levit to include noncompact hyperbolic manifolds \cite{gl-manifold-comm}, where they constructed hyperbolic manifolds from commensurability classes of decorated graphs, thus proving a superexponential growth. However, it was not until 2024 \cite{bdr-incomm-4-5} that the results above were explicitly passed to Coxeter polytopes in dimensions higher than three, where one cannot freely associate a polytope to
    a desired quadratic form. Taking the advantage of Makarov's garland construction, Bogachev, Douba, and Raimbault proved the belief \cite{ft-claim} that the garlands give rise to infinitely many incommensurable compact Coxeter polytopes in dimensions 4 and 5.

    In the second part of the paper, we take advantage of the full list of four-dimensional polytopes with seven facets from \cite{mz-4d-7f} and our new list of five-dimensional polytopes with eight facets. We believe that further complete classifications of the subclasses of hyperbolic Coxeter polytopes will be useful for constructing infinitely many commensurability classes in even higher dimensions.
    Families of noncompact Coxeter polytopes of finite volumes containing infinitely many incommensurable classes are constructed via two methods. Our first method builds polytopes from nonarithmetic one-cusped blocks in dimensions four and five, utilizing the nonarithmetic invariants. The second method relies on our generalizations of the results of Bogachev-Douba-Raimbault, building families of polytopes from (quasi-)arithmetic blocks with a count that grows either linearly or exponentially with volume. Along with the additional help of \cite{tumarkin-nplustwo} and \cite{roberts2016}, we will prove the following commensurability results.
    \begin{theorem}\label{thm:main}   
        There are infinitely many incommensurable noncompact Coxeter polytopes of finite volumes in $\H^n$ for $n=4, 5, 6, 7, 9$. Moreover, there exists a constant $c(n)>1$ such that for all $V\gg 1$, the number of incommensurable classes of noncompact $n$-dimensional polytopes with volume less than $V$ is at least $c(n)^V$.
    \end{theorem}
    Note that the polytopes to be constructed are all nonarithmetic. Indeed, it is still an open question that whether there are infinitely many commensurability classes of quasi-arithmetic Coxeter polytopes, regardless of their compactness.

    We start with preliminaries and the algorithm classifying all five-dimensional polytopes with eight facets in \cref{sec:58}. The new list of complete classification sits in \cref{sec:appendix}. Commensurability invariants and gluing arguments will then be either repeated or generalized in \cref{sec:comm-inv}. In \cref{sec:dimension-four} and \cref{sec:five-dimensional-family}, we will demonstrate our first method of gluing nonarithmetic blocks, partly proving \cref{thm:main}. At last, in \cref{sec:bdr-noncompact}, noncompact analogs of Bogachev-Doub-Raimbault's arguments will be applied to pairs of polytopes discovered in existing lists of polytopes or our new list, completing the proof of the main theorem. Throughout the paper, we will freely use the terms Coxeter polytopes, their reflection groups, and their corresponding orbifolds with no distinction.

    \vspace*{0.3cm}
    
    \textbf{Acknowledgments.} We thank Jean Raimbault for helpful correspondence and valuable comments concerning the proof of \cref{prop:bdr-noncompact}.

    \section{Preliminary} \label{section:cchp}
    In this section, we recall some essential facts about Coxeter hyperbolic polytopes, including Gram matrices, Coxeter diagrams, characterization theorems, etc. Readers can refer to, for example, \cite{Vinberg:1993} for more details.

    For a hyperbolic acute-angled $n$-polytope $P=\cap_{k\in \mathcal{I}} H_{k}^-$, where $H_i^{-}$is the negative half-space bounded by the hyperplane $H_i$ in $\mathbb{H}^n$ with normal vectors pointing outwards, we consider its \textit{Gram matrix} $G(P)$, where $G(P)_{ij}=\langle e_i,e_j\rangle$. Here $e_i$ are vectors from the set $\{e_k\in \mathbb{E}^{n,1}: i\in\mathcal{I}\}$ that determines the hyperplanes $H_k^-$. Many of the combinatorial, geometric, and arithmetic properties of $P$ can be inferred from $G(P)$. For instance, the entry $\langle e_i,e_j\rangle$ characterizes the configuration of bounding hyperplanes as follows:
    \[\langle e_i,e_j\rangle=\begin{cases}
        1, &\text{if }j=i,\\
        -\cos\alpha_{ij}, &\text{if }H_i, H_j\text{ intersect at a diahedral angle }\alpha_{ij},\\
        -\cosh\rho_{ij}, &\text{if }H_i, H_j\text{ are ultraparallel with a distance }\rho_{ij},\\
        -1, &\text{if }H_i, H_j\text{ are parallel.}
    \end{cases}\]
    
    For a Coxeter polytope, it is convenient to represent it by a decorated graph, which is known as the \textit{Coxeter graph}, denoted by $\Sigma=\Sigma(P)$. Every node $i$ in $\Sigma$ corresponds a bounding hyperplane $H_i$ of $P$. Two nodes $i_1$ and $i_2$ corresponding to hyperplanes $H_i$ and $H_j$ are joined (1) by a solid edge with a weight $2\leq k_{ij} < \infty$ if the hyperplanes intersect in $\mathbb{H}^n$ at angle $\frac{\pi}{k_{ij}}$; (2) by a bold edge if the planes are parallel; (3) by a dotted edge, sometimes labeled by $\cosh \rho_{ij}$, if they are ultraparallel with distance $\rho_{ij}$. We omit all edges of weight two. Edges with weights $3$, $4$, $5$, and $6$ are drawn as simple, double, triple, and quadruple edges, respectively. Moreover, the \textit{order} of diagram $\Sigma$, denoted as $|\Sigma|$, refers to the total number of nodes it contains. The \textit{signature} and \textit{rank} of diagram $\Sigma$ correspond to the signature and rank, respectively, of the matrix $G(\Sigma)$.

    A \textit{permutation} of a square matrix refers to a rearrangement of its rows combined with the same rearrangement of its columns. A square matrix $M$ that cannot be represented as a direct sum of two matrices up to permutations is said to be \textit{indecomposable}. Every matrix can be represented uniquely as a direct sum of indecomposable matrices, which are called (indecomposable) components. We say a polytope is \textit{indecomposable} if its Gram matrix $G(P)$ is indecomposable. Note that a non-degenerate hyperbolic polytope $P$ is decomposable if it has a proper face $F$ that is orthogonal to every hyperplane which does not contain it. The orthogonal projection onto the face $F$ establishes a fibration of $P$ into polyhedral cones, ensuring that all vertices of $P$ reside within $F$. Consequently, every convex hyperbolic polytope of finite volume is indecomposable.
    
    \begin{figure}[h]
    	\scalebox{0.26}[0.26]{\includegraphics {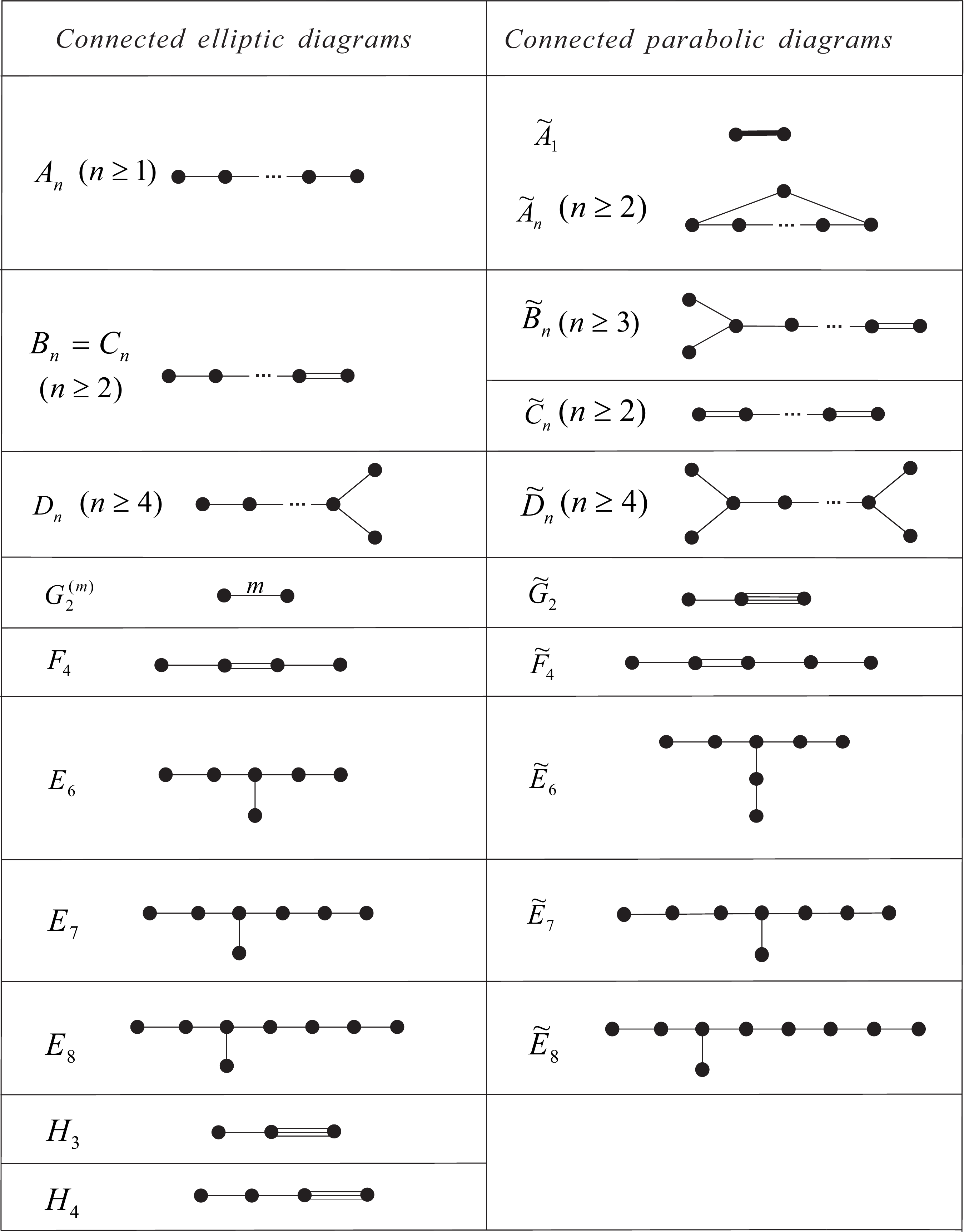}}
    	\caption{Connected elliptic (left) and connected parabolic (right) Coxeter diagrams.}
    	\label{figure:coxeter}
    \end{figure}

    In 1907, Perron found a remarkable property of eigenvalues and eigenvectors of matrices with positive entries. Frobenius later generalized it by investigating the spectral properties of indecomposable non-negative matrices.
    
    \begin{theorem}[Perron-Frobenius, \cite{G:1959}]\label{thm:PF}
    	An indecomposable real matrix $A=(a_{ij})$ with non-negative entries always possesses a maximal eigenvalue $r$, which is positive, simple, and the corresponding eigenvector has all positive coordinates.
    \end{theorem}
    
    The Gram matrices $G(P)$ of an indecomposable Coxeter polytope are symmetric matrices with non-positive entries off the diagonal, and all diagonal elements are equal to $1$. Applying \cref{thm:PF} to $I-G(P)$, we can identify a unique smallest eigenvalue of $G(P)$, denoted by $\lambda$. In the case of $G(P)$ being semi-definite, the deficiency of an indecomposable semi-positive definite matrix $G(P)$ does not exceed $1$, and any proper submatrix of $G(P)$ is positive definite. For a Coxeter polytope $P$, its Coxeter diagram $\Sigma(P)$ is said to be \textit{elliptic} if $G(P)$ is positive definite; $\Sigma (P)$ is called \textit{parabolic} if every indecomposable component of $G(P)$ is degenerate and all proper subdiagrams of each indecomposable component are elliptic. The elliptic and connected parabolic diagrams are proved to be the Coxeter diagrams of spherical and Euclidean Coxeter simplices, respectively. They are classified by Coxeter in \cite{Coxeter:1934} as shown in \cref{figure:coxeter}.
    
    A connected diagram $\Sigma$ is referred to as a \textit{Lann\'{e}r diagram} if it is neither elliptic nor parabolic, and every proper subdiagram of $\Sigma$ is elliptic. These diagrams represent compact hyperbolic Coxeter simplices. The complete list of such diagrams was first reported by Lann\'{e}r \cite{Lanner:1950}, and the lists for two- and three-dimensional cases are provided in \cref{figure:lanner}. Similarly, noncompact hyperbolic simplices of finite volume can also be enumerated. They are known as quasi-Lann\'{e}r diagrams, and characterized by connected non-elliptic and non-parabolic Coxeter graphs where proper subgraphs are either elliptic or connected parabolic graphs.  The two- and three-dimensional lists of these simplices are presented in \cref{figure:lanner}.
    
    \begin{figure}[h]
    	\scalebox{0.31}[0.31]{\includegraphics {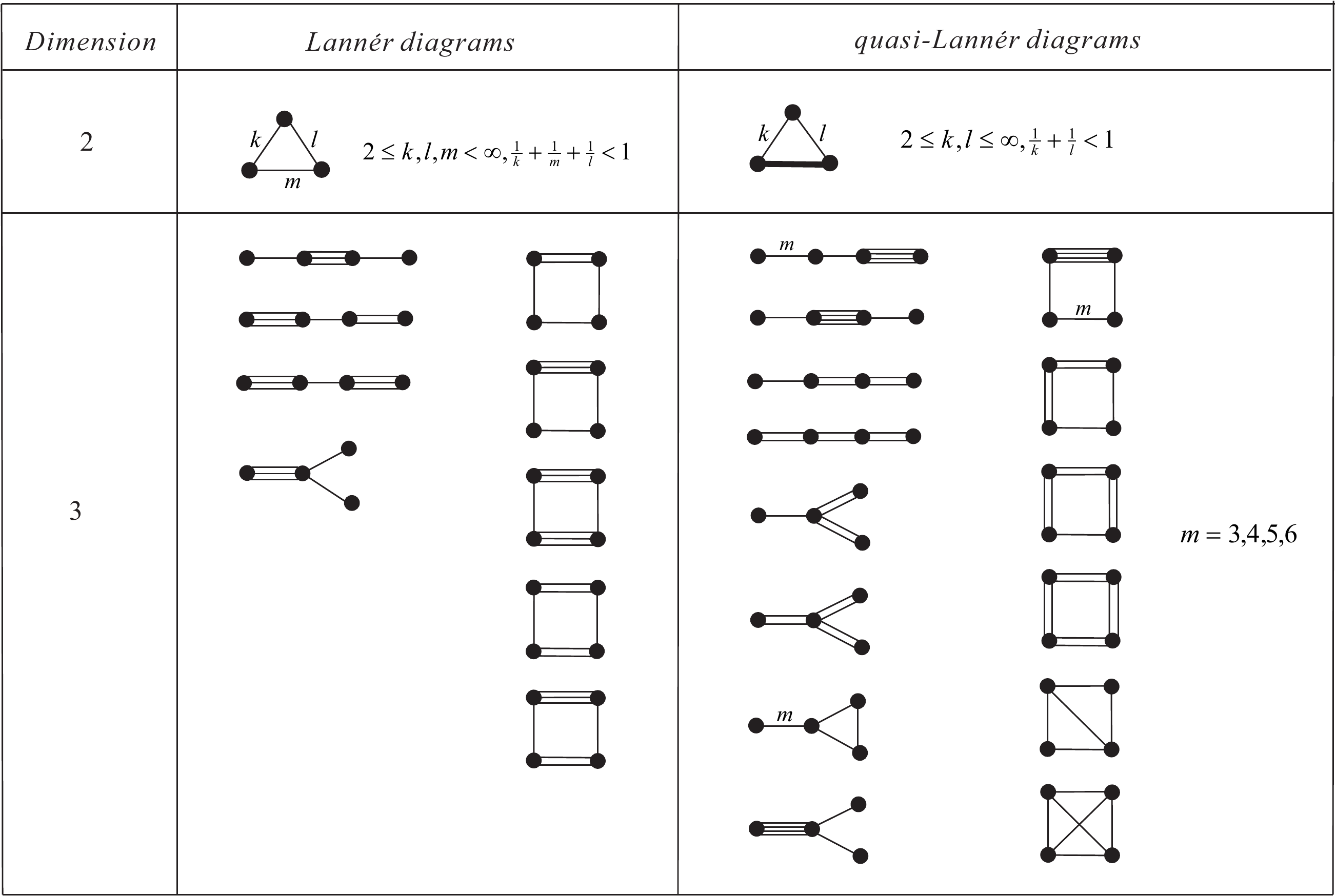}}
    	\caption{The Lann\'{e}r and quasi-Lann\'{e}r diagrams correspond to two- or three-dimensional hyperbolic Coxeter simplices, respectively.}
    	\label{figure:lanner}
    \end{figure}

    Although the full list of hyperbolic Coxeter polytopes remains incomplete, basic properties of Gram matrices and polytopes were already discovered by Vinberg in \cite{vinberg-1985}:
    \begin{theorem}[Theorem 2.1 in \cite{vinberg-1985}] \label{thm:signature}
        Let $G=(g_{ij})$ be an indecomposable symmetric matrix of signature $(n,1)$, where $g_{ii}=1$ and $g_{ij}\leq 0$ if $i\ne j$. Then there exists a unique, up to isometry, convex hyperbolic polytope $P\subseteq\H^n$, whose Gram matrix coincides with $G$.
    \end{theorem}
    \begin{theorem}[Theorem 2.1 and 3.2 in \cite{vinberg-1985}] \label{Vinberg:thm3.1} 
    	Let $P=\mathop{\cap}\limits_{i\in I} H_i^- \subseteq \H^n$ be an acute-angled polytope and let $G=G(P)$ be its Gram matrix. Denote by $G_J$ the principal submatrix of $G$ with rows and columns $J\subseteq I$. Then, 
    	\begin{enumerate}[\normalfont(i)]
    		\item The intersection $\cap_{j\in J}H_j^-\subseteq \H^n$ is a face $F$ of $P$ if and only if the matrix $G_J$ is positive definite;
    		\item The intersection $\overline{\cap_{j\in J}H_j^-}$ is an ideal vertex of $P$ if and only if the matrix $G_J$ is a parabolic matrix of rank $n-1$.
    	\end{enumerate}
    \end{theorem}
    \begin{theorem}[Theorem 4.2 in \cite{vinberg-1985}] \label{Vinberg:thm4.2} 
    	Let $P$ be an acute-angled polytope and $G$ its Gram matrix. Let $\mathcal{F}$ (resp. $\widetilde{\mathcal{F}}$) be the collection of elliptic submatrices (resp. elliptic and parabolic submatrices) of $G$. Partially order $\mathcal{F}$ (resp. $\widetilde{\mathcal{F}}$) by submatrix relations.
    	Then, $P$ is compact (resp. finite-volume) if and only if any of the following holds:
    	\begin{enumerate}[\normalfont(i)]
    		\item The polytope $P$ contains at least one ordinary vertex (resp. ordinary or ideal vertex). For every ordinary vertex (resp. vertex) of $P$ and any edge of $P$ emanating from it, there is precisely one other ordinary vertex (resp. vertex) of $P$ on that edge.
    		\item The partially ordered set $\mathcal{F}$ (resp. $\widetilde{\mathcal{F}}$) is isomorphic to the poset of some $d$-dimensional abstract combinatorial polytope.
    	\end{enumerate}
    \end{theorem}

    To classify all five-dimensional hyperbolic Coxeter polytopes with eight facets, we took a two-stage approach. We start by systematically enumerating all potential matrices over admissible $5$-polytopes with eight facets, ensuring they are elliptic around every ordinary vertex and parabolic around every ideal vertex. To speed up the algorithm, we strategically introduce additional geometric constraints during this stage, which serve to refine our search criteria without compromising the comprehensiveness of our classification. In the second stage, we search for polytopes with an indecomposable signature $(n, 1)$ Gram matrix (here $n=5$), thereby producing a full list of the desired polytopes.
    
    \section{Five-dimensional hyperbolic Coxeter polytopes with eight facets}\label{sec:58}
    As we discussed in Introduction, classifications of hyperbolic polytopes are generally unknown. But they are known in several cases where the difference $m-n$ between the number of facets $m$ and the dimension $n$ is small. For $m-n=1$, Lannér classified the compact hyperbolic Coxeter simplices \cite{Lanner:1950}, while several authors classified the noncompact finite-volume simplices \cite{Bou68,vinberg-criterion, Kos67}. For $m-n=2$, Kaplinskaya described the finite-volume Coxeter simplicial prisms \cite{Kap74}; Esselmann later enumerated the remaining compact cases, now known as Esselmann polytopes \cite{Ess96}, and Tumarkin classified the remaining noncompact finite-volume cases \cite{tumarkin-nplustwo}.  For $m-n=3$, Esselmann proved that compact hyperbolic Coxeter $n$-polytopes with $n+3$ facets exist only for $n\le 8$ \cite{Ess94}, and Tumarkin completed their classification \cite{Tum07}. For the noncompact finite-volume case, Tumarkin proved non-existence in dimensions $n\ge 17$, with a unique example in dimension 16 \cite{Tum04b,Tum03}. He also classified pyramids over products of three simplices, which occur only in dimensions $4,5,\ldots,9$ and $13$. The full noncompact finite-volume classification for $m-n=3$ remains open, although Roberts obtained a list in the non-pyramidal case with exactly one non-simple vertex \cite{roberts2016}. 
    
    In search of the gluing block candidates for incommensurable classes, we exhaustively examined all candidates in the above classifications. By a candidate, we mean a pair of Coxeter polytopes whose combinatorial type is that of a truncated simplex and whose truncation facets are codimension-one hyperbolic Coxeter simplices orthogonal to all adjacent non-truncation facets. More specifically, in truncation number one, that is, within the family of $n$-polytopes with $n+2$ facets, a suitable facet must allow doubling. In other words, the dihedral angles between this facet and the facets adjacent to it must be of the form $\pi/m$, with $m$ even, so that the doubled polytope again satisfies the Coxeter condition and produces two orthogonal gluing facets. For truncation number two, one expects both truncation facets to be orthogonal codimension-one simplices, each suitable for gluing to the corresponding facet of its partner.
    
    We identified an appropriate pair of four-dimensional polytopes in the recent census of the second and fourth authors \cite{mz-4d-7f}, namely in the list of $331$ four-dimensional finite-volume hyperbolic Coxeter polytopes with seven facets. This observation motivated us to construct the five-dimensional eight-facet census to address the existence of infinitely many commensurability classes in dimension five. The complete list of Coxeter diagrams can be found in \cref{sec:appendix}. In this section, we will briefly explain the entire algorithm\footnote[2]{Codes available at \url{https://github.com/GeoTopChristy/HCPd}}, with pseudo-codes available in \cref{alg:enumeration}.

    \subsection{Admissible combinatorial types}
    The classification of combinatorial types of $n$-polytopes with $m$ facets, where $m\leq n+3$, has been accomplished using Gale diagrams. If the vertex link of a vertex $v$ in an $n$-polytope is an $(n-1)$-simplex, then $v$ is referred to as a \textit{simple} vertex. Otherwise, the vertex is \textit{non-simple}. An $n$-polytope is termed \textit{simple} if every vertex is simple. For simple $n$-polytopes with $m$ facets, there is a counting formula by Perles, see for example \cite{G:1967} and also \cite{BD:1998}. We obtain the data of $5$-polytopes with eight facets from \cite{FMM:13}. The numbers $0$, $1$, $\cdots$, $7$ are used to label the seven facets in loc. cit., with each line signifying a combinatorial type. Each \textit{facet bracket} corresponds to a vertex that is intersected by facets with labels contained in the bracket. For instance, for the polytope $P_1$ in \cref{table:combinatoricSelction58}, there are $13$ vertices. To be clearer, the first bracket indicates a vertex formed by the intersection of the facets $F_1$, $F_2$, $F_3$, $F_4$, $F_5$, $F_6$, and $F_7$.

	According to \cref{Vinberg:thm3.1}, the subgraphs formed by nodes that correspond to facets intersecting at ideal points of five-dimensional hyperbolic Coxeter polytopes yield rank four parabolic graphs. These diagrams and basic related combinatorial data are listed in \cref{table:rank3Euclidean}. We say a polytope is \textit{admissible} if its vertex links are either $4$-simplices $D_0=\Delta^4$, simplicial $4$-prisms $D_1=\Delta^1\times \Delta^3$, $D_2=\Delta^2\times\Delta^2$, or $D_3=\Delta^2\times\Delta^1\times\Delta^1$. The four polytopes are simple. Among the total of $116$ five-dimensional polytopes with eight facets, we found $39$ admissible polytopes as shown in \cref{table:combinatoricSelction58}. For each vertex whose link is one of the above polytopes, we amalgamate the facets of the intersection complex, giving a single simplex, with commas separating the distinct product factors of simplices. For instance, the first vertex of $P_{28}$ is non-simple: its link is a simplicial $4$-prism, where the facets $F_2$ and $F_7$ yield a $1$-simplex and the facets $F_3, F_4, F_5, F_6$ yield a $3$-simplex. This vertex is thus recorded as $[27, 3456]$.
	
		\begin{table}[H]
		{\scriptsize
			\begin{tabular}{c|c|c|c|c}
				\Xcline{1-5}{1.2pt}
				\hline
				Rank four & \rule{0pt}{12pt} 
                \multirow{2}{*}{Type $A, B, C, D, F$} & 
                $\widetilde{A}_1\sqcup S_i$&
				$S_1\sqcup S_2$, 
                &$S_i\sqcup\widetilde{A}_1\sqcup\widetilde{A}_1$\\
                Parabolic diagrams & 
                &$S_i\in\{\widetilde{A}_3,\widetilde{B}_3,\widetilde{C}_3\}$&
                 $S_i\in\{\widetilde{A}_2,\widetilde{C}_2,\widetilde{G}_2\}$&
                 $S_i\in\{\widetilde{A}_2,\widetilde{C}_2,\widetilde{G}_2\}$
                 \\

                \hline
                
				Combinatorial type & $\Delta^4$& 
                $\Delta^1\times \Delta^3$ & $\Delta^2\times\Delta^2$ & $\Delta^2\times\Delta^1\times\Delta^1$
                \\
				\hline
				\# Facets &5 & 6 & 6& 7\\
				\hline
				\# Vertices & 5& 8 & 9 & 12\\
				
				\Xcline{1-5}{1.2pt}
			\end{tabular}
		}
		
		\caption{Rank $4$ parabolic Coxeter diagrams.}
		\label{table:rank3Euclidean}
	\end{table}

    \begin{center}
    \renewcommand\arraystretch{1.5}
    {\tiny
    \begin{longtable}{c|p{0.88\textwidth}}
    \Xcline{1-2}{1.2pt}
    Label & Facet brackets\\
    \hline
    \endfirsthead
    \Xcline{1-2}{1.2pt}
    Label & Facet brackets\\
    \hline
    \endhead
    1 & [345,  17,  26] [7,6,5,4,0] [7,6,5,3,0] [7,6,4,3,0] [7,5,4,2,0] [7,5,3,2,0] [7,4,3,2,0] [6,5,4,1,0] [6,5,3,1,0] [6,4,3,1,0] [5,4,2,1,0] [5,3,2,1,0] [4,3,2,1,0]\\
    \hline
    2 & [257,  346] [057,  146] [037,  126] [034,  125] [7,6,5,3,1] [7,6,4,2,0] [7,5,4,3,1] [7,4,3,2,1] [7,4,2,1,0] [6,5,4,2,0] [6,5,3,2,0] [6,5,3,1,0]\\
    \hline
    3 & [257,  346] [057,  146] [037,  126] [7,6,5,3,1] [7,6,4,2,0] [7,5,4,3,1] [7,4,3,2,1] [7,4,2,1,0] [6,5,4,2,0] [6,5,3,2,0] [6,5,3,1,0] [5,4,3,2,1] [5,4,2,1,0] [5,3,2,1,0]\\
    \hline
    4 & [3467,  25] [057,  146] [037,  126] [7,6,5,3,1] [7,6,4,2,0] [7,5,4,3,1] [7,4,3,2,1] [7,4,2,1,0] [6,5,4,3,0] [6,5,3,1,0] [6,4,3,2,0] [5,4,3,1,0] [4,3,2,1,0]\\
    \hline
    5 & [3467,  25] [1467,  05] [037,  126] [7,6,5,3,1] [7,6,4,2,0] [7,5,4,3,1] [7,4,3,2,1] [7,4,2,1,0] [6,5,4,3,1] [6,4,3,2,0] [6,4,3,1,0] [4,3,2,1,0]\\
    \hline
    6 & 3467,  25] [1467,  05] [1267,  03] [7,6,5,3,1] [7,6,4,2,0] [7,5,4,3,1] [7,4,3,2,1] [7,4,2,1,0] [6,5,4,3,1] [6,4,3,2,1] [6,4,2,1,0]\\
    \hline
    7 & [267, 345] [057, 146] [7,6,5,3,1] [7,6,4,3,0] [7,6,3,1,0] [7,5,4,2,1] [7,5,3,2,1] [7,4,3,2,1] [7,4,3,1,0] [6,5,4,2,0] [6,5,3,2,1] [6,5,2,1,0] [6,4,3,2,0] [6,3,2,1,0] [5,4,2,1,0] [4,3,2,1,0]\\
    \hline
    8 & [267, 345] [067, 145] [7,6,5,3,1] [7,6,4,3,1] [7,5,4,2,0] [7,5,3,2,1] [7,5,2,1,0] [7,4,3,2,1] [7,4,2,1,0] [6,5,4,2,0] [6,5,3,2,0] [6,5,3,1,0] [6,4,3,2,0] [6,4,3,1,0] [5,3,2,1,0] [4,3,2,1,0]\\
    \hline
    9 & [267, 345] [067, 145]  [7,6,5,3,1] [7,6,4,3,1] [7,5,4,2,0] [7,5,3,2,1] [7,5,2,1,0] [7,4,3,2,1] [7,4,2,1,0] [6,5,4,2,0] [6,5,3,2,1] [6,5,2,1,0] [6,4,3,2,0] [6,4,3,1,0] [6,3,2,1,0] [4,3,2,1,0]\\
    \hline
    10 & [267, 345] [067, 145] [7,6,5,3,1] [7,6,4,3,1] [7,5,4,2,0] [7,5,3,2,1] [7,5,2,1,0] [7,4,3,2,1] [7,4,2,1,0] [6,5,4,2,0] [6,5,3,2,1] [6,5,2,1,0] [6,4,3,2,1] [6,4,2,1,0]\\
    \hline
    11 & [27, 3456] [067, 145] [7,6,5,3,1] [7,6,4,3,1] [7,5,4,3,0] [7,5,3,1,0] [7,4,3,1,0] [6,5,4,2,0] [6,5,3,2,0] [6,5,3,1,0] [6,4,3,2,0] [6,4,3,1,0] [5,4,3,2,0]\\
    \hline
    12 & [27, 3456] [067, 145] [7,6,5,3,1] [7,6,4,3,1] [7,5,4,3,0] [7,5,3,1,0] [7,4,3,1,0] [6,5,4,2,0] [6,5,3,2,1] [6,5,2,1,0] [6,4,3,2,0] [6,4,3,1,0] [6,3,2,1,0] [5,4,3,2,0] [5,3,2,1,0]\\
    \hline
    13 & [27, 3456] [067, 145] [7,6,5,3,1] [7,6,4,3,1] [7,5,4,3,0] [7,5,3,1,0] [7,4,3,1,0] [6,5,4,2,0] [6,5,3,2,1] [6,5,2,1,0] [6,4,3,2,1] [6,4,2,1,0] [5,4,3,2,0] [5,3,2,1,0] [4,3,2,1,0]\\
    \hline
    14 & [27, 3456] [06, 1457] [7,6,5,3,1] [7,6,4,3,1] [7,5,4,3,0] [7,5,3,1,0] [7,4,3,1,0] [6,5,4,2,1] [6,5,3,2,1] [6,4,3,2,1] [5,4,3,2,0] [5,4,2,1,0] [5,3,2,1,0] [4,3,2,1,0]\\
    \hline
    15 & [27, 3456] [06, 1457] [7,6,5,3,1] [7,6,4,3,1] [7,5,4,3,0] [7,5,3,1,0] [7,4,3,1,0] [6,5,4,2,1] [6,5,3,2,1] [6,4,3,2,1] [5,4,3,2,1] [5,4,3,1,0]\\
    \hline
    16 & [27, 3456]] [07, 1456] [7,6,5,3,1] [7,6,4,3,1] [7,5,4,3,1] [6,5,4,2,0] [6,5,3,2,1] [6,5,2,1,0] [6,4,3,2,1] [6,4,2,1,0] [5,4,3,2,0] [5,4,3,1,0] [5,3,2,1,0] [4,3,2,1,0]\\
    \hline
    17 & [27, 3456]] [07, 1456] [7,6,5,3,1] [7,6,4,3,1] [7,5,4,3,1] [6,5,4,2,0] [6,5,3,2,1] [6,5,2,1,0] [6,4,3,2,1] [6,4,2,1,0] [5,4,3,2,1] [5,4,2,1,0]\\
    \hline
    18 & [27, 3456] [7,6,5,4,1] [7,6,5,3,1] [7,6,4,3,0] [7,6,4,1,0] [7,6,3,1,0] [7,5,4,3,0] [7,5,4,1,0] [7,5,3,1,0] [6,5,4,2,1] [6,5,3,2,1] [6,4,3,2,0] [6,4,2,1,0] [6,3,2,1,0] [5,4,3,2,0] [5,4,2,1,0] [5,3,2,1,0]\\
    \hline
    19 & [267, 345] [7,6,5,4,1] [7,6,5,3,1] [7,6,4,3,1] [7,5,4,2,1] [7,5,3,2,0] [7,5,3,1,0] [7,5,2,1,0] [7,4,3,2,0] [7,4,3,1,0] [7,4,2,1,0] [6,5,4,2,1] [6,5,3,2,0] [6,5,3,1,0] [6,5,2,1,0] [6,4,3,2,0] [6,4,3,1,0] [6,4,2,1,0]\\
    \hline
    20 & [267, 345] [7,6,5,4,1] [7,6,5,3,1] [7,6,4,3,1] [7,5,4,2,1] [7,5,3,2,1] [7,4,3,2,0] [7,4,3,1,0] [7,4,2,1,0] [7,3,2,1,0] [6,5,4,2,0] [6,5,4,1,0] [6,5,3,2,0] [6,5,3,1,0] [6,4,3,2,0] [6,4,3,1,0] [5,4,2,1,0] [5,3,2,1,0]\\
    \hline
    21 & [267, 345] [7,6,5,4,1] [7,6,5,3,1] [7,6,4,3,1] [7,5,4,2,1] [7,5,3,2,1] [7,4,3,2,0] [7,4,3,1,0] [7,4,2,1,0] [7,3,2,1,0] [6,5,4,2,1] [6,5,3,2,0] [6,5,3,1,0] [6,5,2,1,0] [6,4,3,2,0] [6,4,3,1,0] [6,4,2,1,0] [5,3,2,1,0]\\
    \hline
    22 & [267, 345] [7,6,5,4,1] [7,6,5,3,1] [7,6,4,3,1] [7,5,4,2,1] [7,5,3,2,1] [7,4,3,2,1] [6,5,4,2,0] [6,5,4,1,0] [6,5,3,2,0] [6,5,3,1,0] [6,4,3,2,0] [6,4,3,1,0] [5,4,2,1,0] [5,3,2,1,0] [4,3,2,1,0]\\
    \hline
    23 & [267, 345] [7,6,5,4,1] [7,6,5,3,1] [7,6,4,3,1] [7,5,4,2,1] [7,5,3,2,1] [7,4,3,2,1] [6,5,4,2,1] [6,5,3,2,0] [6,5,3,1,0] [6,5,2,1,0] [6,4,3,2,0] [6,4,3,1,0] [6,4,2,1,0] [5,3,2,1,0] [4,3,2,1,0]\\
    \hline
    24 & [267, 345] [7,6,5,4,1] [7,6,5,3,1] [7,6,4,3,1] [7,5,4,2,1] [7,5,3,2,1] [7,4,3,2,1] [6,5,4,2,1] [6,5,3,2,1] [6,4,3,2,0] [6,4,3,1,0] [6,4,2,1,0] [6,3,2,1,0] [4,3,2,1,0]\\
    \hline
    25 & [27, 3456] [7,6,5,4,1] [7,6,5,3,1] [7,6,4,3,1] [7,5,4,3,0] [7,5,4,1,0] [7,5,3,1,0] [7,4,3,1,0] [6,5,4,2,1] [6,5,3,2,0] [6,5,3,1,0] [6,5,2,1,0] [6,4,3,2,0] [6,4,3,1,0] [6,4,2,1,0] [5,4,3,2,0] [5,4,2,1,0]\\
    \hline
    26 & [27, 3456] [7,6,5,4,1] [7,6,5,3,1] [7,6,4,3,1] [7,5,4,3,0] [7,5,4,1,0] [7,5,3,1,0] [7,4,3,1,0] [6,5,4,2,1] [6,5,3,2,1] [6,4,3,2,0] [6,4,3,1,0] [6,4,2,1,0] [6,3,2,1,0] [5,4,3,2,0] [5,4,2,1,0] [5,3,2,1,0]\\
    \hline
    27 & [27, 3456] [7,6,5,4,1] [7,6,5,3,1] [7,6,4,3,1] [7,5,4,3,0] [7,5,4,1,0] [7,5,3,1,0] [7,4,3,1,0] [6,5,4,2,1] [6,5,3,2,1] [6,4,3,2,1] [5,4,3,2,0] [5,4,2,1,0] [5,3,2,1,0] [4,3,2,1,0]\\
    \hline
    28 & [27, 3456] [7,6,5,4,1] [7,6,5,3,1] [7,6,4,3,1] [7,5,4,3,1] [6,5,4,2,0] [6,5,4,1,0] [6,5,3,2,0] [6,5,3,1,0] [6,4,3,2,0] [6,4,3,1,0] [5,4,3,2,0] [5,4,3,1,0]\\
    \hline
    29 & [27, 3456] [7,6,5,4,1] [7,6,5,3,1] [7,6,4,3,1] [7,5,4,3,1] [6,5,4,2,1] [6,5,3,2,0] [6,5,3,1,0] [6,5,2,1,0] [6,4,3,2,0] [6,4,3,1,0] [6,4,2,1,0] [5,4,3,2,0] [5,4,3,1,0] [5,4,2,1,0]\\
    \hline
    30 & [27, 3456] [7,6,5,4,1] [7,6,5,3,1] [7,6,4,3,1] [7,5,4,3,1] [6,5,4,2,1] [6,5,3,2,1] [6,4,3,2,0] [6,4,3,1,0] [6,4,2,1,0] [6,3,2,1,0] [5,4,3,2,0] [5,4,3,1,0] [5,4,2,1,0] [5,3,2,1,0]\\
    \hline
    31 & [27, 3456] [7,6,5,4,1] [7,6,5,3,1] [7,6,4,3,1] [7,5,4,3,1] [6,5,4,2,1] [6,5,3,2,1] [6,4,3,2,1] [5,4,3,2,0] [5,4,3,1,0] [5,4,2,1,0] [5,3,2,1,0] [4,3,2,1,0]\\
    \hline
    32 & [7,6,5,4,3] [7,6,5,4,2] [7,6,5,3,2] [7,6,4,3,1] [7,6,4,2,1] [7,6,3,2,1] [7,5,4,3,1] [7,5,4,2,1] [7,5,3,2,1] [6,5,4,3,0] [6,5,4,2,0] [6,5,3,2,0] [6,4,3,1,0] [6,4,2,1,0] [6,3,2,1,0] [5,4,3,1,0] [5,4,2,1,0] [5,3,2,1,0]\\
    \hline
    33 & [7,6,5,4,3] [7,6,5,4,2] [7,6,5,3,2] [7,6,4,3,2] [7,5,4,3,1] [7,5,4,2,1] [7,5,3,2,0] [7,5,3,1,0] [7,5,2,1,0] [7,4,3,2,0] [7,4,3,1,0] [7,4,2,1,0] [6,5,4,3,1] [6,5,4,2,1] [6,5,3,2,0] [6,5,3,1,0] [6,5,2,1,0] [6,4,3,2,0] [6,4,3,1,0] [6,4,2,1,0]\\
    \hline
    34 & [7,6,5,4,3] [7,6,5,4,2] [7,6,5,3,2] [7,6,4,3,2] [7,5,4,3,1] [7,5,4,2,1] [7,5,3,2,1] [7,4,3,2,0] [7,4,3,1,0] [7,4,2,1,0] [7,3,2,1,0] [6,5,4,3,1] [6,5,4,2,1] [6,5,3,2,0] [6,5,3,1,0] [6,5,2,1,0] [6,4,3,2,0] [6,4,3,1,0] [6,4,2,1,0] [5,3,2,1,0]\\
    \hline
    35 & [7,6,5,4,3] [7,6,5,4,2] [7,6,5,3,2] [7,6,4,3,2] [7,5,4,3,1] [7,5,4,2,1] [7,5,3,2,1] [7,4,3,2,1] [6,5,4,3,0] [6,5,4,2,0] [6,5,3,2,0] [6,4,3,2,0] [5,4,3,1,0] [5,4,2,1,0] [5,3,2,1,0] [4,3,2,1,0]\\
    \hline
    36 & [7,6,5,4,3] [7,6,5,4,2] [7,6,5,3,2] [7,6,4,3,2] [7,5,4,3,1] [7,5,4,2,1] [7,5,3,2,1] [7,4,3,2,1] [6,5,4,3,1] [6,5,4,2,0] [6,5,4,1,0] [6,5,3,2,0] [6,5,3,1,0] [6,4,3,2,0] [6,4,3,1,0] [5,4,2,1,0] [5,3,2,1,0] [4,3,2,1,0]\\
    \hline
    37 & [7,6,5,4,3] [7,6,5,4,2] [7,6,5,3,2] [7,6,4,3,2] [7,5,4,3,1] [7,5,4,2,1] [7,5,3,2,1] [7,4,3,2,1] [6,5,4,3,1] [6,5,4,2,1] [6,5,3,2,0] [6,5,3,1,0] [6,5,2,1,0] [6,4,3,2,0] [6,4,3,1,0] [6,4,2,1,0] [5,3,2,1,0] [4,3,2,1,0]\\
    \hline
    38 & [7,6,5,4,3] [7,6,5,4,2] [7,6,5,3,2] [7,6,4,3,2] [7,5,4,3,2] [6,5,4,3,1] [6,5,4,2,1] [6,5,3,2,1] [6,4,3,2,0] [6,4,3,1,0] [6,4,2,1,0] [6,3,2,1,0] [5,4,3,2,0] [5,4,3,1,0] [5,4,2,1,0] [5,3,2,1,0]\\
    \hline
    39 & [7,6,5,4,3] [7,6,5,4,2] [7,6,5,3,2] [7,6,4,3,2] [7,5,4,3,2] [6,5,4,3,1] [6,5,4,2,1] [6,5,3,2,1] [6,4,3,2,1] [5,4,3,2,0] [5,4,3,1,0] [5,4,2,1,0] [5,3,2,1,0] [4,3,2,1,0]\\
    \hline
    \Xcline{1-2}{1.2pt}
    \caption{The 39 admissible combinatorial types.}
    \label{table:combinatoricSelction58}
    \end{longtable}}
    \end{center}
    
    \subsection{Hyperbolically admissible Gram matrices}
    A Gram matrix $G$ of order $8$ is called hyperbolically admissible, also termed SELCper matrix in \cite{mz-4d-7f}, in dimension five if it satisfies all following conditions:
    \begin{enumerate}
        \item If $I \subseteq \{1,\dots,8\}$ is the set of facets incident to an ordinary vertex, then the induced Coxeter diagram on $I$ is spherical and rank $5$.
        \item If $J$ is the set of facets incident to an ideal vertex, then $G_J$ is positive semidefinite of rank 4.
        \item The elliptic and parabolic subdiagrams recover the face poset of the prescribed $5$-polytope, with no spurious facets or missing faces.
        \item If $J$ is the set of facets whose Gram submatrix corresponds to a $k$-simplex type, then $G_J$ is a Lannér diagram if all involved facets are free of ideal vertices, or a quasi-Lannér diagram otherwise.
    \end{enumerate}
    
    We will search for hyperbolically admissible matrices over each admissible combinatorial type of $5$-dimensional polytopes with $8$ facets. To narrow down the possibilities, we apply the weak signature condition: namely, every $7 \times 7$ minor $M_i$ satisfies $\det(M_i) = 0$. Finally, we carry out an if-and-only-if screening using the signature and the form of the matrix entries: all off-diagonal entries contributed by intersecting faces must equal $-\cos(\pi/m_{ij})$. This entire procedure is completely analogous to the treatment of the four-dimensional seven facets case in \cite{mz-4d-7f}, to which the reader is referred for further details. The final enumeration results are presented in \cref{table:result}, with the last two rows for \cite{roberts2016}, \cite{Tum07} and \cite{Tum04b}. The list of all Coxeter diagrams is available in \cref{sec:appendix}.

    \begin{table}[H]
    \centering
    {\scriptsize
    \renewcommand{\arraystretch}{1.45}
    \begin{tabular}{c|c|ccc|cccccc|cccc|cc}
    \Xcline{1-17}{1.2pt}
    
    \multicolumn{1}{c|}{\# Non-simple vertices}
    & 4
    & \multicolumn{3}{c|}{3}
    & \multicolumn{6}{c|}{2}
    & \multicolumn{4}{c|}{1}
    & \multicolumn{2}{c}{0}\\
    \Xcline{1-17}{1.2pt}
    
    \multicolumn{1}{c|}{Polytope labels}
    & 2 & 3 & 4 & 6 & 7 & 8 & 10 & 11 & 15 & 17 & 1 & 22 & 27 & 28 & 37 & 39\\
    \Xcline{1-17}{1.2pt}
    
    \multicolumn{1}{c|}{\# Finite-volume, noncompact}
    & 2 & 2 & 1 & 2 & 2 & 2 & 20 & 2 & 3 & 10 & 15 & 6 & 1 & 54 & 1 & 2\\
    \hline
    
    \multicolumn{1}{c|}{\# Compact}
    & & & & & & & & & & & & & & & 1 & 15\\
    \hline
    
    \makecell[c]{Non-pyramidal with\\ one non-simple vertex}
    & & & & & & & & & & & & 6 & 1 & 54 & & \\
    \hline
    
    \makecell[c]{Compact and pyramidal}
    & & & & & & & & & & & 15 & & & & 1 & 15\\
    \Xcline{1-17}{1.2pt}
    
    \end{tabular}
    }
    
    \hspace*{0.5cm}
    \caption{Statistics of the result. Specifically, for a compact hyperbolic Coxeter polytope to be considered, it must be simple, and thus a combinatorial polytope with non-simple vertices cannot possess the structure of a compact hyperbolic Coxeter polytope as shown blank on line $4$ and line $6$.}
    \label{table:result}
    \end{table}

    \subsection{Validation}		
    
    Hyperbolic Coxeter non-pyramidal $n$-polytopes with $n+3$ facets and exactly one non-simple vertex were previously enumerated by Roberts \cite{roberts2016}. In dimension $n=5$, our classification recovers Roberts's $61$ hyperbolic Coxeter polytopes, occurring on the combinatorial types $P_{22}$, $P_{27}$, and $P_{28}$.
    
    Furthermore, Tumarkin classified non-simple pyramidal Coxeter polytopes in \cite{Tum03} and compact hyperbolic Coxeter $n$-polytopes with $n+3$ facets in \cite{Tum04b}. In dimension five, our results agree with Tumarkin's classification; the corresponding examples are of combinatorial types $P_1$, $P_{37}$ and $P_{39}$.
    \begin{algorithm}[H]
    \caption{Enumeration of hyperbolic Coxeter $5$-polytopes with eight facets}
    \label{alg:enumeration}
    \begin{algorithmic}[1]
    
    \Require The $116$ combinatorial types of $5$-polytopes with eight facets
    \Ensure The set $\mathcal H$ of their finite-volume hyperbolic Coxeter Gram matrices
    
    \State $\mathcal L\gets
    \{\Delta^4,\,
    \Delta^1\times\Delta^3,\,
    \Delta^2\times\Delta^2,\,
    \Delta^2\times\Delta^1\times\Delta^1\}$
    \State $\mathcal A\gets\varnothing$
    
    \ForAll{combinatorial types $P$}
        \State compute the vertex links of $P$
        \If{$\operatorname{link}_P(v)\in\mathcal L$ for every vertex $v$}
            \State $\mathcal A\gets\mathcal A\cup\{P\}$
        \EndIf
    \EndFor
    
    \State $\mathcal H\gets\varnothing$
    
    \ForAll{$P\in\mathcal A$}
        \State assign rank-$5$ spherical diagrams to ordinary vertices
        \State assign rank-$4$ parabolic diagrams to ideal vertices
        \State encode weights greater than $7$ by the placeholder $7$
        \State assemble compatible global diagrams
        \State discard diagrams violating geometric constraints
        \Statex \hspace{\algorithmicindent}
            $\triangleright$ \textit{Lannér, quasi-Lannér, and parabolic-square obstructions}
        \State discard spurious spherical and parabolic subdiagrams
        \State discard decomposable diagrams and duplicates
        \Statex \hspace{\algorithmicindent}
            $\triangleright$ \textit{Duplicates are identified up to facet relabeling.}
        \State let $\mathcal D(P)$ be the remaining diagrams
    
        \ForAll{$D\in\mathcal D(P)$}
            \State construct the symbolic Gram matrix $G=(g_{ij})$
            \State set $g_{ii}=1$
            \State replace placeholder weights by angle variables $a_{ij}$
            \State introduce variables $g_{ij}$ for nonintersecting facets
            \State solve $\det G_{\widehat{i}}=0$ for $i=1,\ldots,8$
    
            \Statex \hspace{\algorithmicindent}
                $\triangleright$ \textit{Coxeter-admissible means
                $a_{ij}=-\cos(\pi/m_{ij})$ and $g_{ij}\leq-1$.}
    
            \ForAll{real solutions $G$}
                \If{$\operatorname{sig}(G)=(5,1,2)$ and $G$ is Coxeter-admissible}
                    \State $\mathcal H\gets\mathcal H\cup\{G\}$
                \EndIf
            \EndFor
        \EndFor
    \EndFor
    
    \State \Return $\mathcal H$
    
    \end{algorithmic}
    \end{algorithm}

    \newpage
    \section{Commensurability invariant}
    \label{sec:comm-inv}
    \subsection{Quasi-arithmetic invariants}
    Given an $n$-dimensional hyperbolic Coxeter polytope $P$ of finite volume with $d$ facets, denote by $k(P)$ and $K(P)$ the Vinberg field and the field generated by entries of the Gram matrix $G(P)=[g_{ij}]_{i,j\leq d}$ of $P$, respectively. A \textit{cyclic product} in the Gram matrix is a product of the form $g_{i_0, i_1}g_{i_1,i_2}\cdots g_{i_{m-1}, i_m}g_{i_m, i_0}$ for some indices $i_0,..., i_m$. The Vinberg field $k(P)$ is therefore by definition the field generated by all cyclic products in $2G(P)$. For simplicity, we will take Vinberg's criterion as our definition of arithmeticity:
    \begin{theoremdefinition}[\cite{vinberg-criterion}]
        A finite-volume hyperbolic Coxeter polytope $P$ is \textup{arithmetic} if and only if
        \begin{enumerate}[\normalfont(i)]
            \item The field $K(P)$ is totally real.
            \item For any real embedding $\s:K(P)\hookrightarrow \R$ nontrivial on $k(P)$, the matrix $\s(G(P))$ is positive semi-definite.
            \item All cyclic products in $2G(P)$ are algebraic integers in $k(P)$.
        \end{enumerate}
        A polytope is said to be \textup{quasi-arithmetic} if the first two are satisfied. It is further said to be \textup{properly quasi-arithmetic} if the last one does not hold.
    \end{theoremdefinition}
    Let $Q(P)$ be the Vinberg form associated to $P$ is the restriction of the Lorentzian product, on the hyperboloid model $\H^n\subseteq \R^{n, 1}$ with normal unit vectors $\{e_i\}$, to the vector space over $k(P)$ spanned by vectors of the form
    \[v_1=2e_1,\quad v_{i_1,\dots,i_m}=2^mg_{1,i_1}g_{i_1,i_2}\cdots g_{i_{m-1}, i_m}e_{i_m}.\] 
    Two quadratic forms $(V_1, Q_1)$ and $(V_2, Q_2)$ over some field $k$ are \textit{similar} if there is an isomorphism $\varphi$ from $V_1$ to $V_2$ and a scalar $\lambda\in k^\times$ such that $q_1$ is isomorphic to $\lambda q_2\circ\varphi$ as quadratic forms on $V_1$. We restate a commensurability criterion first shown by Gromov and Piatetski-Shapiro in \cite{gps-important} and generalized by Dotti \cite{dotti-good-stuff}.
    \begin{theorem}\label{thm:arith-inv}
    Two finite-volume Coxeter polytopes are commensurable only if their Vinberg fields coincide, and their associated Vinberg forms are similar. Furthermore, if the polytopes are arithmetic, then the converse is true.
    \end{theorem}

    The arithmetic invariants in \cref{thm:arith-inv} are increasingly difficult to use when polytopes have more facets. We will repeat important results from Bogachev-Douba-Raimbault. The following statements from their paper are, respectively, Corollary~2.1.1 and Lemma~3.2, and the resulting commensurability criterion obtained by combining these results.
    
    Let $L$ be a semisimple real Lie group and $\Gamma$ a lattice in $L$. The \textit{adjoint trace field} of $\Gamma$ is the subfield $\langle\tr(\operatorname{adj} g):g\in\Gamma)\rangle\subseteq\R$. If a hyperbolic Coxeter polytope $P$ with reflection group $\Gamma$ has finite volume, then $k(P)$ coincides with $k(\Gamma)$. The \textit{ambient group} of $\Gamma$ in $L$ is the unique up to $k(\Gamma)$-isogeny $k(\Gamma)$-group $G$ such that $G(\R)$ is isogenous to $L$ and $G(k(\Gamma))$ virtually contains $\Gamma$ under this isogeny. 
    \begin{corollary}\label{cor:same-ambient}
        If $\Gamma_1$ and $\Gamma_2$ are quasi-arithmetic lattices in $L$ and $\Gamma_1\cap\Gamma_2$ is Zariski-dense in $L$, then the ambient groups of $\Gamma_1$ and $\Gamma_2$ coincide.
    \end{corollary}
    \begin{lemma}\label{lem:qsiarith-inv}
        Let $k$ be a field of characteristic $0$, let $m\geq2$, and let $Q_1, Q_2$ be two nondegenerate quadratic forms on $k^m$. Then $PO(Q_1)$ is $k$-isogenous to $PO(Q_2)$ if and only if $Q_1, Q_2$ are similar over $k$, i.e., if and only if there exists $\lambda\in k^\times$ such that $Q_1$ and $\lambda Q_2$ are isometric over $k$.
    \end{lemma}
   \begin{remark}
        In particular, we will test for distinct ambient groups of quasi-arithmetic polytopes by comparing the determinant square classes (in odd dimensions) and Hasse invariants of the quadratic forms of the polytopes.
    \end{remark}
    
    \subsection{Nonarithmetic invariants for noncompact polytopes}
    The superfluous commensurability invariants of arithmetic polytopes are no longer effective in a process of gluing, as the resulting polytopes are usually nonarithmetic and the invariants become very difficult to compute. We thus recall another weaker invariant for nonarithmetic, which proves to be particularly effective for certain families constructed in this paper.
    
    A \textit{horoball packing} in $\H^n$ is a set of horoballs with disjoint interiors. A \textit{cusp neighborhood} of a hyperbolic orbifold $O=\H^n/\Gamma$ is simply a connected open subset that lifts to a horoball packing in $\H^n$. A \textit{cusp} is therefore the end determined by equivalent cusp neighborhoods, that is, neighborhoods that contain a common smaller cusp neighborhood. By definition, cusps exist only if the orbifold is noncompact; we further assume the orbifold has finite volume.
    \begin{lemma}[\cite{ghh-cusped-comm}]
        Two finite volume hyperbolic cusped orbifolds cover a common orbifold if and only if each of them admits cusp neighborhoods lifting to isometric horoball packings in $\H^n$.
    \end{lemma}
    \begin{proof}
        We refer to Lemmas 2.1, 2.2, and 2.3 in \cite{ghh-cusped-comm} for the proof. Indeed, if two orbifolds admit neighborhoods lifting to isometric horoball packings in $\H^n$, then they cover $\H^n/\Gamma$ where the symmetric groups of the horoball packings generate $\Gamma$. This group is discrete, so the lemmas in loc. cit. apply.
    \end{proof}
    A cusp neighborhood is said to be \textit{maximal} if no cusp neighborhood properly contains it. If an orbifold $O$ has only one cusp, then there is a unique maximal cusp neighborhood $N$ around that cusp. We thus define the \textit{cusp density} of $O$ to be the ratio $\cd(O)=\vol(N)/\vol(O)$. Finite covers preserve cusp densities. Note that if we assume the orbifolds are nonarithmetic, by Margulis's theorem, their commensurators are discrete. In this case, the orbifolds are commensurable if and only if they cover a common orbifold. Thus, the lemma has the following corollary.
    \begin{corollary}\label{cor:cd-incomm}
        If two nonarithmetic hyperbolic one-cusped orbifolds of finite volumes are commensurable, then they have the same cusp density.
    \end{corollary}

    \subsection{Commensurability in gluing nonarithmetic blocks}
    A facet $F$ of a Coxeter polytope $P$ is called \textit{admissible} if every facet of $P$ that meets $F$ along a codimension-two face is orthogonal to $F$. Equivalently, in the Coxeter diagram of $P$, every node corresponding to a facet incident to $F$ is not joined by a solid edge to the node corresponding to $F$. Suppose $P_1, P_2$ are two $n$-dimensional polytopes such that $P_i$ contains a facet $F_i$, and $F_1$ is isometric to $F_2$ via some isometry $\varphi$. Suppose the two polytopes can be glued together along $F_1\simeq_\varphi F_2$, and form a new Coxeter polytope. Then a Coxeter polytope of the same dimension can be obtained by gluing $P_1$ and $P_2$ along $\varphi$. We denote by $P_1\cup_\varphi P_2$ the resulting polytope. It is immediate that $\vol(P_1\cup_\varphi P_2)=\vol(P_1)+\vol(P_2)$.
    \begin{lemma}
    \label{lem:decreasing-density}
        Suppose $P_1$ is a one-cusped polytope with finite volume and $P_2$ is a compact polytope. If the unique maximal cusp neighborhood $N(P_1)$ of $P_1$ has a closure not intersecting the gluing facet $F_1$, then we have the following:
        \[\vol N(P_1)=\vol N(P_1\cup_\varphi P_2).\]
        Here $N(P_1\cup_\varphi P_2)$ is the unique maximal cusp neighborhood of the new one-cusped polytope. Furthermore,
        \[\operatorname{cd}(P_1)>\operatorname{cd}(P_1\cup_{\varphi} P_2).\]
    \end{lemma}
    \begin{proof}
        The second claim follows immediately from the first and the fact that $\vol(P_1\cup_\varphi P_2)=\vol(P_1)+\vol(P_2)>\vol(P_1)$. 

        We first note that under the assumption of this lemma, the gluing facet $F_1\simeq_\varphi F_2$ is compact. Thus $X=P_1\cup_\varphi P_2$ has exactly one cusp $v$, the one coming from the ideal vertex of $P_1$. Let $\Gamma_1$ and $\Gamma_X$ be the reflection groups of $P_1$ and $X$, respectively. Let $\Gamma_v$ be the parabolic reflection subgroup generated by the reflections in the facets containing $v$. The facets containing $v$ are unchanged by the gluing, so the same group $\Gamma_v$ is the cusp subgroup for both $P_1$ and $X$. Choose a horoball $B\subseteq\H^n$, based at $v$, such that $N(P_1)=B\cap P_1$. 
        
        Consider the cusp chambers $\mathcal C_1=\cup_{\gamma\in\Gamma_v}\gamma P_1$ and $\mathcal C_X=\cup_{\gamma\in\Gamma_v}\gamma X$. The boundary of $\mathcal C_1$ is formed by the $\Gamma_v$-translates of the facets of $P_1$ that do not contain $v$. Since $N(P_1)$ is maximal, $\partial B$ is tangent to at least one boundary wall $\gamma E\subseteq\partial\mathcal C_1$, where $E$ is a facet of $P_1$ not containing $v$. Every element of $\Gamma_v$ preserves $B$. Hence the hypothesis $\overline{N(P_1)}\cap F_1=\varnothing$ implies $\overline B\cap\gamma F_1=\varnothing$ for every $\gamma\in\Gamma_v$. It follows that the tangent wall $\gamma E$ cannot be a translate of the gluing facet. Thus $E\neq F_1$. The additional part of $\mathcal C_X$ is obtained by attaching a copy of $P_2$ across each wall $\gamma F_1$. We claim that $B$ does not enter any of these attached copies. 
        
        Suppose, to the contrary, that $B$ meets some copy $\gamma P_2$. Choose $x\in B\cap\operatorname{int}(\gamma P_1)$ sufficiently close to $v$, and choose $y\in B\cap\gamma P_2$. Both $B$ and $\gamma\bigl(P_1\cup_{\varphi}P_2\bigr)$ are convex. Therefore the geodesic segment $[x,y]$ lies in $B$ and must cross the common facet $\gamma F_1$. This gives $B\cap\gamma F_1\neq\varnothing$, contradicting the preceding paragraph. Consequently, $B\cap\mathcal C_X=B\cap\mathcal C_1$, and in particular $B\cap X=B\cap P_1$. Thus $B$ determines a cusp neighborhood of $X$ having the same volume as $N(P_1)$. It remains to prove maximality. The facet $E\neq F_1$ survives as a boundary facet of $X$. Hence $\gamma E$ remains a reflecting wall for $\Gamma_X$, and $\partial B$ remains tangent to $\gamma E$. Let $B'$ be a horoball based at $v$ that properly contains $B$. Since $\partial B$ is tangent to $\gamma E$, the horoball $B'$ crosses $\gamma E$. Reflection in $\gamma E$ therefore sends $B'$ to a horoball whose interior overlaps the interior of $B'$. Thus the $\Gamma_X$-translates of $B'$ do not form a horoball packing, so $B'$ cannot define an embedded cusp neighborhood of $X$. It follows that $N(X)=B\cap X$ is the maximal cusp neighborhood of $X$. Therefore $\vol N(X)=\vol N(P_1)$.
    \end{proof}

    \subsection{Commensurability in gluing quasi-arithmetic blocks}
    \label{sec:gluing-qsi-arith}
    In this section, we show that results from \cite{bdr-incomm-4-5} still hold for noncompact polytopes. While their commensurability invariant criteria are applicable regardless of compactness, the gluing lemmas in their paper have to be adjusted with care. We start with an orbifold, quasi-arithmetic analog of Lemma~3.5 in \cite{gl-manifold-comm}.
    \begin{lemma}[orbifold version of Lemma 3.5 in \cite{gl-manifold-comm}]\label{lem:gl-comm}
        Let $X_1$ and $X_2$ be two suborbifolds with totally geodesic boundary contained in two quasi-arithmetic hyperbolic $n$-orbifolds whose ambient groups are not $k$-isogenous, $k$ being the trace field of both orbifolds. Assume either that $\partial X_1$ and $\partial X_2$ are compact and $n\geq 3$, or that $\partial X_1$ and $\partial X_2$ have finite volume and $n\geq 4$. If $O$ is a complete hyperbolic $n$-orbifold and $U_1, U_2\hookrightarrow O$ are embedded suborbifolds with totally geodesic boundary admitting finite isometric orbifold covers $p_i:U_i\to X_i$ for $i=1, 2$, then $U_1\cap U_2$ has empty interior.
    \end{lemma}
    \begin{proof}
        Let $M_i=\Gamma_i\backslash\mathbb H^n$ be the quasi-arithmetic hyperbolic orbifold containing $X_i$, and let $G_i$ be the ambient $k$-group of $\Gamma_i$. Suppose that $U_1\cap U_2$ has nonempty interior. Since $M_i$ has finite volume and $X_i\subseteq M_i$ is embedded, $X_i$ has finite volume. Since $p_i:U_i\to X_i$ is a finite isometric orbifold covering, $U_i$ also has finite volume, and $\partial U_i$ is compact or has finite volume according to the corresponding hypothesis on $\partial X_i$.

        The geometric part of the proof of Lemma 3.5 in \cite{gl-manifold-comm} now applies to $U_1$ and $U_2$. More precisely, its proof shows, using only the finite-volume and boundary hypotheses, that if $U_1\cap U_2$ contains an open set, then the monodromy of the orbifold fundamental group of some connected component of $U_1\cap U_2$ is Zariski-dense in $SO(n,1)$. This part of its proof applies unchanged to orbifolds: one replaces fundamental groups by orbifold fundamental groups and, whenever necessary, passes to torsion-free finite-index subgroups. Zariski density and the relevant finite-index conclusions are unaffected by this passage.

        Let $\Delta<SO(n,1)$ be the resulting Zariski-dense monodromy group. For each $i$, the composition $U_i\to X_i\hookrightarrow M_i$ is a local isometry. By uniqueness of developing maps, there exists $g_i\in\operatorname{Isom}(\mathbb H^n)$ such that $g_i\Delta g_i^{-1}\subseteq\Gamma_i$. Consequently, after replacing $\Gamma_i$ by the conjugate $\Gamma_i'=g_i^{-1}\Gamma_i g_i$, we have $\Delta\subseteq\Gamma_1'\cap\Gamma_2'$. In particular, $\Gamma_1'\cap\Gamma_2'$ is Zariski-dense in $SO(n,1)$. The lattices $\Gamma_1'$ and $\Gamma_2'$ are quasi-arithmetic and have the same ambient groups, up to $k$-isogeny, as $\Gamma_1$ and $\Gamma_2$. \cref{cor:same-ambient} therefore implies that $G_1$ and $G_2$ are $k$-isogenous. This contradicts the hypothesis. Hence $U_1\cap U_2$ has empty interior.
    \end{proof}
    \begin{remark}
        We note that Lemma 3.5 in \cite{gl-manifold-comm} is stated for arithmetic manifolds, where they only require incommensurability. This result still holds for quasi-arithmetic manifolds with distinct ambient groups. Indeed, the commensurability criterion of \cite{gps-important} quoted at the beginning of their proof of Lemma 3.5 can be replaced by \cref{cor:same-ambient}, and everything else follows verbatim.
    \end{remark}

    We are now ready to replicate the results of \cite{bdr-incomm-4-5}. The following two sets of assumptions are for convenience.
    \begin{assumption}\label{ass:we-can}
        Assume $P_1$ and $P_2$ are two quasi-arithmetic Coxeter polytopes in $\H^n$, $n\geq 4$, with finite volume and distinct ambient groups. Suppose $P_1$ has one admissible facet, and $P_2$ contains two admissible facets, all isometric to the same Coxeter polytope $A$ in $\H^{n-1}$. Assume further that the two facets in $P_2$ are nonadjacent. By adjacent, we mean the two facets meet at a face of codimension two, so having closures merely joining at a cusp is allowed.
    \end{assumption}
    \begin{assumption}\label{ass:we-can-now}
        Assume $P_1$ and $P_2$ are two quasi-arithmetic Coxeter polytopes in $\H^n$, $n\geq 4$, with finite volume and distinct ambient groups. Suppose each $P_i$ contains two admissible facets, both isometric to the same Coxeter polytope $A$ in $\H^{n-1}$. Assume further that the two facets are nonadjacent.
    \end{assumption}
    The following results are straightforward analogs of Proposition~4.1, 4.2 from Bogachev-Douba-Raimbaul.
    \begin{proposition}\label{prop:bdr-linear}
        Let $P_1, P_2$ be a pair of polytopes satisfying \cref{ass:we-can}. For $m\geq 1$, cap one $P_1$ block to $m$ copies of $P_2$ glued together. Then the resulting polytopes $L_m$ are pairwise incommensurable.
    \end{proposition}
    \begin{proposition}\label{prop:bdr-noncompact}
        Let $P_1, P_2$ be a pair of polytopes satisfying \cref{ass:we-can-now}. Given a word $\alpha$, let $X_\alpha$ be obtained by gluing the pieces $P_{\alpha_i}$ along their common boundary components. For any $m\geq 1$ and any words $\alpha, \beta\in\{1, 2\}^m$ such that $\Gamma_{X_\alpha}$ is commensurable with $\Gamma_{X_\beta}$, we have that $\beta$ or $\bar{\beta}$ is a subsequence of $(\alpha, \bar{\alpha})$. Here $\bar\alpha$ and $\bar\beta$ are mirrors of $\alpha$ and $\beta$ resp.
    \end{proposition}
    \begin{remark}
        It is immediate from the above proposition that if a pair of $n$-dimensional polytopes satisfying \cref{ass:we-can-now} exists, then the commensurability classes grow at least exponentially wrt volumes.
    \end{remark}
    We first state a general version of the propositions above. Let $A$ be a finite-volume hyperbolic $(n-1)$-orbifold. Assume either that $A$ is compact and $n\geq 3$, or that $A$ is noncompact of finite volume and $n\geq 4$. Let $P_1, \ldots, P_r$ be finite-volume quasi-arithmetic hyperbolic $n$-orbifolds with totally geodesic boundary. Assume that for each $a$ there are two distinguished embedded totally geodesic boundary facets $A_a^{\pm}$, both isometric to $A$. We allow the closures of $A_a^-$ and $A_a^+$ to meet in cusps, but require that $A_a^-\cap A_a^+$ has empty interior and that, in the hyperbolic orbifold $P_a$, they are distinct boundary facets. Assume moreover that gluing $A_a^+$ to $A_b^-$ by the prescribed isometry produces a complete finite-volume hyperbolic orbifold. Assume that each $P_a$ is contained as a suborbifold with totally geodesic boundary in a finite-volume quasi-arithmetic hyperbolic $n$-orbifold $\Sigma_a$, that the trace field of each $\Sigma_a$ is $k$, and that the ambient groups of the $\Sigma_a$ are pairwise not $k$-isogenous. For a word $\alpha\in\{1, \ldots, r\}^{\Z/m\Z}$, let $X_\alpha$ be obtained by gluing $A_{\alpha_i}^+\subseteq P_{\alpha_i}$ to $A_{\alpha_{i+1}}^-\subseteq P_{\alpha_{i+1}}$ for every $i\in\Z/m\Z$.
    \begin{proposition}[finite-volume orbifold version of Proposition~2.1 in \cite{raimbault-manifold-comm}]\label{prop:gluing-comm}
        Fix two words $\alpha, \beta$ If $X_\alpha$ and $X_\beta$ are commensurable, then there exists $p\in\mathbb Z/m\mathbb Z$ such that either $\beta_j=\alpha_{j+p}$ for all $j$, or $\beta_j=\alpha_{p-j}$ for all $j$. Equivalently, $\beta$ or its mirror $\bar\beta$ lies in the $\mathbb Z/m\mathbb Z$-shift orbit of $\alpha$.
    \end{proposition}
    \begin{proof}
        Let $O$ be a common finite orbifold cover of $X_\alpha$ and $X_\beta$, with covering maps $\pi:O\to X_\alpha$ and $\pi':O\to X_\beta$. We first claim that the lifts of pieces of different labels have intersection with empty interior. Indeed, if $U_a$ is a connected component of the lift of a $P_a$-piece and $U_b$ is a connected component of the lift of a $P_b$-piece with $a\neq b$, then $U_a$ and $U_b$ admit finite isometric orbifold covers to $P_a$ and $P_b$, respectively. Since $P_a$ and $P_b$ lie inside quasi-arithmetic orbifolds $\Sigma_a$ and $\Sigma_b$ whose ambient groups are not $k$-isogenous, the preceding lemma implies that $U_a\cap U_b$ has empty interior. Now let $\alpha_i\neq\alpha_{i+1}=\cdots=\alpha_j\neq\alpha_{j+1}$ be a maximal block of constant label in $\alpha$, and let $C$ be a connected component of $\pi^{-1}(P_{\alpha_{i+1}, \ldots, \alpha_j})$. By the claim, $C$ can meet only lifts of $\beta$-pieces with the same label in nonempty interior, so $C$ is contained in a connected component $C'$ of the lift of some maximal block $P_{\beta_{i'+1}, \ldots, \beta_{j'}}$ with the same label. The same claim applied at the two ends of the block forbids $C'$ from extending past $C$, hence $C=C'$. Thus maximal constant-label blocks in $\alpha$ and $\beta$ correspond. Their lengths are equal: if a common finite orbifold cover covers a block of $s$ copies of $P_a$ with degree $d_s$ and a block of $t$ copies of $P_a$ with degree $d_t$, then comparing boundary volumes gives $2d_s\operatorname{vol}(A)=2d_t\operatorname{vol}(A)$, so $d_s=d_t$, and comparing total volumes gives $d_s s\operatorname{vol}(P_a)=d_t t\operatorname{vol}(P_a)$, hence $s=t$. Therefore the cyclic decompositions of $\alpha$ and $\beta$ into maximal constant-label blocks agree in label and length. The same combinatorial argument as in \cite{raimbault-manifold-comm} then implies that the two cyclic words agree up to cyclic shift or reversal, namely $\beta_j=\alpha_{j+p}$ for all $j$, or $\beta_j=\alpha_{p-j}$ for all $j$.
    \end{proof}
    \begin{proof}[Proof of \cref{prop:bdr-linear}]
        This is evident if one modifies the first part of the proof of \cref{prop:gluing-comm} for two different $L_m$ and $L_{m'}$, just like the original proof in \cite{bdr-incomm-4-5}.
    \end{proof}
    \begin{proof}[Proof of \cref{prop:bdr-noncompact}]
        The proof is identical to that of the original statement, replacing its use of Lemma 3.2 by \cref{lem:gl-comm} and Proposition 2.1 in \cite{raimbault-manifold-comm} by \cref{prop:gluing-comm}. One subtlety to note is that the surjective morphism $\varphi$ in the original proof is still well-defined when the two isometric gluing facets meet at the cusp: the order of generator product is infinity, and thus making it a homomorphism. 
    \end{proof}

    \section{The four-dimensional one-cusped families}
    \label{sec:dimension-four}
    In this section, we construct infinite sequences of nonarithmetic one-cusped Coxeter polytopes in $\mathbb H^4$ and prove that they are pairwise incommensurable. The construction has two ingredients. The first is a one-cusped Coxeter polytope, which provides the unique cusp neighborhood. The second consists of compact Coxeter blocks, which can be attached along admissible facets without changing the maximal cusp, making \cref{lem:decreasing-density} applicable.
    \subsection{Gluing nonarithmetic polytopes}
    We begin with the four-dimensional Coxeter simplex $\Delta_{16,5}$ (see \cref{figure:delta_165}).  
    \begin{figure}[H]
        \centering
        \includegraphics[
          width=0.25\textwidth
        ]{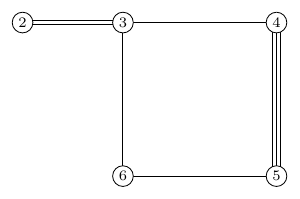}
        \caption{The simplex $\Delta_{16, 5}$.}
        \label{figure:delta_165}
    \end{figure}

    Label its facets by $F_2, F_3, F_4, F_5, F_6$. With respect to this ordering, its Gram matrix is
    \[G(\Delta_{16,5})=
    \begin{pmatrix}
        1 & -\frac{\sqrt2}{2} & 0 & 0 & 0\\
        -\frac{\sqrt2}{2} & 1 & -\frac12 & 0 & -\frac12\\
        0 & -\frac12 & 1 & -\frac{1+\sqrt5}{4} & 0\\
        0 & 0 & -\frac{1+\sqrt5}{4} & 1 & -\frac12\\
        0 & -\frac12 & 0 & -\frac12 & 1
    \end{pmatrix}.\]
    This matrix has signature $(4,1,0)$, and hence determines a hyperbolic Coxeter simplex in $\mathbb H^4$. A direct computation of the signatures of its $4\times 4$ principal submatrices gives
    \[\begin{array}{c|c|c}
        J & \operatorname{sign}(G_J) & \text{type of the corresponding vertex}\\
        \hline
        \{3,4,5,6\} & (3,1,0) & \text{hyperideal }\\
        \{2,4,5,6\} & (4,0,0) & \text{ordinary}\\
        \{2,3,5,6\} & (4,0,0) & \text{ordinary}\\
        \{2,3,4,6\} & (3,0,1) & \text{ideal}\\
        \{2,3,4,5\} & (3,1,0) & \text{hyperideal}
    \end{array}\]

    We truncate the two hyperideal
    vertices
    \[F_3\cap F_4\cap F_5\cap F_6\quad\text{and}\quad F_2\cap F_3\cap F_4\cap F_5.\]
    By Proposition 4.4 in \cite{vinberg-1985} (excision lemma), each new truncation facet is admissible. Notice the resulting Coxeter polytope is precisely $P_{16,5}$ in the list of \cite{mz-4d-7f}, which we display in \cref{figure:p416}. Denote by $P_1^4$ the new polytope for convenience. We label the two new truncation facets by $F_7$ and $F_1$, respectively, so that $F_1$ is orthogonal to all $F_i$ for $i=2, 3, 4, 5$ and $F_7$ is orthogonal to $F_j$ for $j=3, 4, 5, 6$. The polytope $P^4_1$ has exactly one ideal vertex, namely
    \[v=F_2\cap F_3\cap F_4\cap F_6.\]
    All other vertices are ordinary. Hence $P^4_1$ is a finite-volume one-cusped Coxeter polytope in $\mathbb{H}^4$. The facet $F_1$ is admissible. The Coxeter structure induced on $F_1$ is the hyperbolic Coxeter $3$-simplex $[4, 3, 5]$ since the Coxeter diagram on the adjacent facets $F_2, F_3, F_4, F_5$ is the chain with labels $4, 3, 5$.

    With respect to the ordering $F_1,\ldots,F_7$, the Gram matrix of
    $P^4_1$ is
    \[G(P^4_1)=
        \begin{pmatrix}
        1 & 0 & 0 & 0 & 0 & -a & -b\\
        0 & 1 & -\frac{\sqrt2}{2} & 0 & 0 & 0 & -c\\
        0 & -\frac{\sqrt2}{2} & 1 & -\frac12 & 0 & -\frac12 & 0\\
        0 & 0 & -\frac12 & 1 & -\frac{1+\sqrt5}{4} & 0 & 0\\
        0 & 0 & 0 & -\frac{1+\sqrt5}{4} & 1 & -\frac12 & 0\\
        -a & 0 & -\frac12 & 0 & -\frac12 & 1 & 0\\
        -b & -c & 0 & 0 & 0 & 0 & 1
        \end{pmatrix}\]
    By construction, $P^4_1$ is a hyperbolic Coxeter polytope in $\mathbb{H}^4$, so the signature of its Gram matrix must be $(4,1,2)$. Then every $6\times 6$ minors must be degenerate, which helps us solve for $a, b, c$.
    \[a=\frac12\sqrt{11+5\sqrt5},\quad b=3\sqrt{\frac{6+\sqrt5}{31}},\quad c=\sqrt{\frac{17+8\sqrt5}{31}}.\]
    Due to the entry $-c$, $K(P)$ is not totally real, and $P_1^4$ is not quasi-arithmetic. We conclude
    \begin{lemma}
        The polytope $P_1^4$ with an admissible facet $[4,3,5]$ is finite-volume, non-arithmetic, and one-cusped.
    \end{lemma}
    We will use three compact Coxeter $4$-polytopes from the list of \cite{mz-4d-7f}.
    \[P_{16,2},\quad P_{16,3},\quad P_{16,4}.\]
    They are all obtained by truncating two hyperideal vertices of Coxeter simplices. This gives compact Coxeter $4$-polytopes with seven facets. In each case, the facet labelled $F_1$ is isometric, as a Coxeter polytope, to the Coxeter $3$-simplex $[4,3,5]$.

    Moreover, in each of these three compact polytopes, the facet $F_7$ is admissible. Doubling along $F_7$ therefore produces another compact Coxeter $4$-polytope. We denote these doubled polytopes by $P^4_i=D_{F_7}(P_{16, i})$ for $i=2, 3, 4$. Each $P^4_i$, $i=2,3,4$, contains two distinguished admissible facets, coming from the two copies of $F_1$. Both of these facets are isometric to $[4,3,5]$. We write them as $A_i^- , A_i^+ \subseteq P^4_i$. For the one-cusped block $P^4_1$, we simply write $A_1=F_1(P^4_1)$. Thus $A_1, A_i^-, A_i^+$ are all isometric Coxeter $3$-simplices of type
    $[4,3,5]$, and all of them are admissible facets.
    \begin{figure}[H]
        \centering
        \includegraphics[
          width=\textwidth
        ]{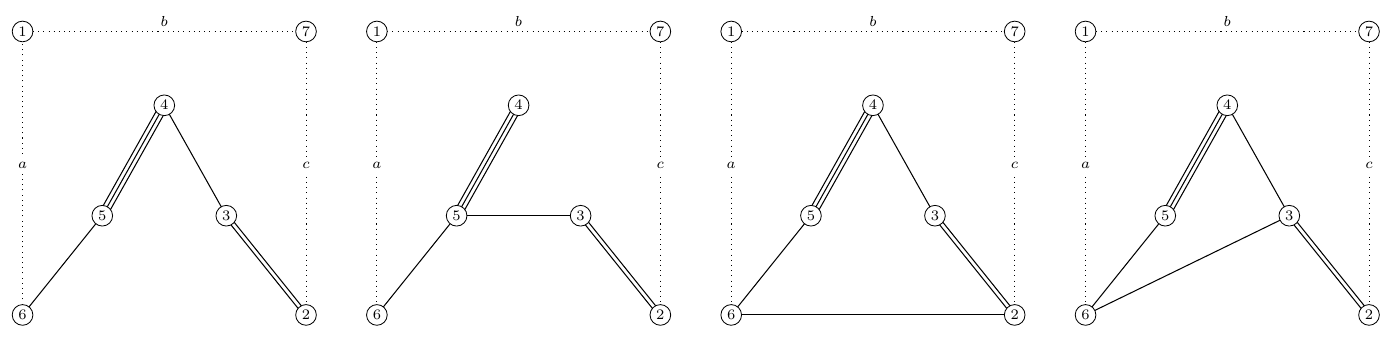}
        \caption{From left to right: $P_{16, 2}$, $P_{16, 3}$, $P_{16, 4}$, and $P_{16, 5}$.}
        \label{figure:p416}
    \end{figure}
    
    Let $\alpha=(\alpha_1,\ldots,\alpha_k)$ be a word in the alphabet $\{2,3,4\}$. We define
    \[X^4_{\alpha}=X^4_{\alpha_1,\ldots,\alpha_k}\]
    to be the Coxeter polytope obtained by gluing
    \[P^4_1, P^4_{\alpha_1}, P^4_{\alpha_2}, \dots, P^4_{\alpha_k}\]
    We glue $A_1$ to $A_{\alpha_1}^{-}$, and $A_{\alpha_j}^{+}$ to $A_{\alpha_{j+1}}^{-}$ for $j=1,\ldots,k-1$ via isometries. Each $X^4_\alpha$ is again a Coxeter polytope. They are also nonarithmetic. Indeed, the facets $F_2$ and $F_7$ of the original block $P^4_1$ survive in $X^4_\alpha$, and their Gram matrix entry is still $-c$. Therefore, $K(X^4_\alpha)$ is not totally real.

    \subsection{Commensurability of the glued polytopes}
    \label{subsec:4d-cusp-radius}
    We now determine the maximal cusp of $P^4_1$.  We first record a
    standard radius formula.
    \begin{lemma}
    \label{lem:cusp-radius}
        Let $P\subseteq \H^n$ be a Coxeter polytope, and let
        \[v=F_{i_1}\cap\cdots\cap F_{i_n}\]
        be an ideal vertex represented in the Lorentz model by a light-like vector $u$. Let $H_j=e_j^\perp$ be a facet hyperplane not passing through $v$. In the upper half-space model in which $v$ is sent to $\infty$, the hyperplane $H_j$ is represented by a Euclidean hemisphere $S_j$. Its Euclidean radius is
        \[r(S_j)=\frac{1}{|\langle u,e_j\rangle|}. \]
    \end{lemma}
    \begin{proof}
        Choose another light-like vector $u'\in \mathbb R^{n,1}$ such that $\langle u,u'\rangle=-1$, and set $E=(\mathbb R u\oplus \mathbb R u')^\perp$. Then the Lorentzian form is positive definite on $E$. Every vector $x\in\mathbb R^{n,1}$ can be written uniquely as
        \[x=\lambda u+\mu u'+w,
        \quad w\in E. \]
       
        The corresponding upper half-space coordinates are
        \[x\longmapsto \left(\frac{w}{\mu},\frac1\mu\right). \]
        The inverse map is
        \[(z,t)\longmapsto \frac{t^2+\|z\|^2}{2t}u+\frac{1}{t}u'+  \frac{1}{t}z.\]    
        Write $e_j=\lambda u+\mu u'+w$. Then $\langle u,e_j\rangle=-\mu$. The hyperplane equation $\langle x,e_j\rangle=0$ becomes
        \[\left\|z-\frac{w}{\mu}\right\|^2+t^2=\frac{1}{\mu^2}.\]
        Thus, the image of $H_j$ is a Euclidean hemisphere of radius
        $1/|\mu|$, and hence
        \[ r(S_j)=\frac{1}{|\langle u,e_j\rangle|}.\]
    \end{proof}
    We apply this lemma to the unique ideal vertex
    \[v=F_2\cap F_3\cap F_4\cap F_6\]
    of $P^4_1$.  In the Lorentz model, this ideal vertex is represented by the light-like vector
    \[u=\sqrt2 e_2+2e_3+e_4+e_6.\]
    Indeed,
    \[\langle u,e_2\rangle=\langle u,e_3\rangle=\langle u,e_4\rangle=\langle u,e_6\rangle=0,\]
    and a direct computation gives $\langle u,u\rangle=0$.
    
    The facets not passing through $v$ are $F_1, F_5, F_7$. Therefore, in the upper half-space model with $v$ sent to $\infty$, these three facets are represented by Euclidean hemispheres $S_1, S_5, S_7$. Using
    \cref{lem:cusp-radius}, we compute
    \[\langle u,e_5\rangle=-\frac{3+\sqrt5}{4}, \quad
    \langle u,e_7\rangle=-\sqrt{\frac{34+16\sqrt5}{31}}, \quad
    \langle u,e_1\rangle=-\sqrt{\frac{11+5\sqrt5}{4}}.\]
    Numerically,
    \[\left|\langle u,e_5\rangle\right|\approx 1.31, \quad
    \left|\langle u,e_7\rangle\right|\approx 1.50, \quad
    \left|\langle u,e_1\rangle\right|\approx 2.35.\]
    which means $r(S_5)>r(S_7)>r(S_1)$.
    \begin{proposition}
    \label{prop:4d-maximal-cusp}
        The maximal cusp neighborhood of $P^4_1$ has its closure tangent to the facet $F_5$. In particular, its closure does not touch the gluing facet $F_1$.
    \end{proposition}
    \begin{proof}
        In the upper half-space model with the ideal vertex $v$ sent to $\infty$, cusp neighborhoods based at $v$ are given by horoballs of the form
        \[B_t=\{(z,s)\in \mathbb R^3\times\mathbb R_{>0}: s\geq t\}.\]
        The maximal cusp is obtained when such a horoball first becomes tangent to one of the hemispheres corresponding to the facets not passing through $v$. Since $r(S_5)>r(S_7)>r(S_1)$, the first tangency occurs at $S_5$. Therefore, the maximal cusp is tangent to $F_5$, and it does not touch $F_1$.
    \end{proof}
    We are now ready to prove the second main theorem in dimension four.
    \begin{theorem}
    \label{thm:4d-infinite-incommensurable}
        Given a sequence of letters $\{\alpha_i\}$ in the alphabet $\{2, 3, 4\}$, define $\alpha^i=(\alpha_1,\dots,\alpha_i)$. The sequence
        \[X^4_{\alpha^1},X^4_{\alpha^2},X^4_{\alpha^3},\ldots\]
        gives infinitely many pairwise incommensurable nonarithmetic one-cusped Coxeter polytopes of finite volume in $\H^4$.
    \end{theorem}
    \begin{proof}
        By construction, each $X^4_{\alpha^i}$ is a finite-volume one-cusped Coxeter polytope, $P_{\alpha_{i+1}}$ is compact, and they have isometric admissible facets. By \cref{lem:decreasing-density}, we see that for all $i>0$,
        \[\cd(X^4_{\alpha^i})>\cd(X^4_{\alpha^i}\cup P_{\alpha_{i+1}})=\cd(X^4_{\alpha^{i+1}}).\]
        By \cref{cor:cd-incomm}, the $X^4_{\alpha^i}$ are pair-wise incommensurable.
    \end{proof}   
    
    \section{The five-dimensional one-cusped families}
    \label{sec:five-dimensional-family}
    In light of the complete list of five-dimensional eight-facet Coxeter polytopes, we can now give the five-dimensional analogue of the previous section. The arguments are identical to the four-dimensional case, so we will only write down the five-dimensional data. Let the capped block to be $P^5_1=P_{39,3}$ in our list, \cref{figure:p57}. We reproduce the labeled diagrams here.
    \begin{figure}[H]
        \centering
        \includegraphics[
          width=0.5\textwidth
        ]{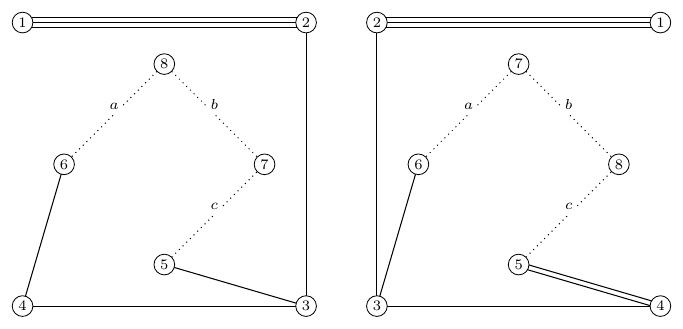}
        \caption{$P_{39, 2}$ (left) and $P_{39, 3}$ (right).}
        \label{figure:p539}
    \end{figure}
    
    \begin{lemma}
    \label{lem:P393-seed}
        The polytope $P_1^5$ is a finite-volume nonarithmetic one-cusped Coxeter polytope in $\H^5$ with an admissible facet $F_8=\left[5,3,3^{1,1}\right]$.
    \end{lemma}
    \begin{proof}
        The Gram matrix of $P_{39,3}$ is 
        \[G(P_1^5)=\begin{pmatrix}
            1 & -\frac{1+\sqrt{5}}{4} & 0 & 0 & 0 & 0 & 0 & 0 \\
            -\frac{1+\sqrt{5}}{4} & 1 & -\frac{1}{2} & 0 & 0 & 0 & 0 & 0 \\
            0 & -\frac{1}{2} & 1 & -\frac{1}{2} & 0 & -\frac{1}{2} & 0 & 0 \\
            0 & 0 & -\frac{1}{2} & 1 & -\frac{\sqrt{2}}{2} & 0 & 0 & 0 \\
            0 & 0 & 0 & -\frac{\sqrt{2}}{2} & 1 & 0 & 0 & -c \\
            0 & 0 & -\frac{1}{2} & 0 & 0 & 1 & -a & 0 \\
            0 & 0 & 0 & 0 & 0 & -a & 1 & -b \\
            0 & 0 & 0 & 0 & -c & 0 & -b & 1
        \end{pmatrix}\]
        where
        \[a=\sqrt{\frac{2+\sqrt{5}}{2}}, \quad b=\frac{\sqrt{5}}{2}, \quad c=\frac{1}{2} \sqrt{2+\sqrt{5}}=\frac{a}{\sqrt{2}}.\]
        The polytope has exactly one ideal vertex,
        \[v=F_2\cap F_3\cap F_4\cap F_5\cap F_6,\]
        and all other vertices are ordinary. Thus $P_1^5$ is finite-volume and one-cusped. It remains to show the polytope is nonarithmetic. The entry $g_{67}=-a$ makes $K(P_1^5)$ not totally real. We take the admissible gluing facet to be $A_1=F_8(P_1^5)$, so $g_{67}$ will survive. The induced Coxeter structure on $A_1$ is $\left[5,3,3^{1,1}\right]$ (see the diagram below for clarity).
        \begin{figure}[H]
            \centering
            \includegraphics[
              width=0.2\textwidth
            ]{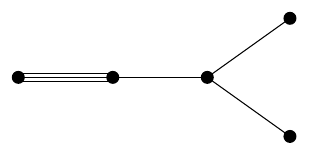}
            \label{figure:p53311}
        \end{figure}
        This completes the proof.
    \end{proof}
    
    Now let $P_{39, 2}$ be the only five-dimensional compact building blocks. Denote by $P_2^5$ the double of $P_{39, 2}$ along its admissible doubling facet $F_7$. The double $P^5_2$ contains two distinguished admissible facets $F_8^\pm$ both isometric to $\left[5,3,3^{1,1}\right]$. Define $X^5_{n}$ for a constant word $(2,2,\dots,2)$ of length $n$ in the same manner as the last section.
    \begin{lemma}
    \label{lem:P393-cusp-not-gluing}
        The closure of the maximal cusp of $P^5_1$ does not intersect the gluing facet $A_1=F_8$.
    \end{lemma}
    \begin{proof}
        The ideal vertex
        \[v=F_2\cap F_3\cap F_4\cap F_5\cap F_6\]
        is represented by the light-like vector
        \[u=e_2+2e_3+2e_4+\sqrt2 e_5+e_6.\]
        The facets not containing $v$ are $F_1,F_7,F_8$. By the cusp-radius formula from \cref{lem:cusp-radius},
        \[r(S_i)=\frac{1}{|\langle u,e_i\rangle|}.\]
        A direct computation gives
        \[\langle u,e_1\rangle=-\frac{1+\sqrt5}{4}, \quad
                \langle u,e_7\rangle=-\sqrt{\frac{2+\sqrt5}{2}}, \quad
                \langle u,e_8\rangle=-\sqrt{\frac{2+\sqrt5}{2}}.\]
        Therefore, $r(S_1)>r(S_7)=r(S_8)$. Hence, the maximal cusp is tangent to $F_1$, and in particular, it is not tangent
        to the gluing facet $F_8$.
    \end{proof}
    \begin{theorem}
    \label{thm:five-dimensional-family}
        Given a constant word $(2,2,\dots,2)$ of length $n$, the sequence
        \[X^5_{1},X^5_{2},X^5_{3},\ldots\]
        gives infinitely many pairwise incommensurable nonarithmetic one-cusped Coxeter polytopes of finite volume in $\H^5$.
    \end{theorem}
    \begin{proof}
        The same proof applies verbatim, mutatis mutandis.
    \end{proof}

    \section{Bogachev-Douba-Raimbault's construction in the noncompact setting}\label{sec:bdr-noncompact}
    In this section, we will construct Makarov's garland (in fact, tiara) in the noncompact setting for dimensions $4, 5, 6, 7$, and $9$. That is, we will prove the following as a direct consequence of \cref{prop:bdr-noncompact}:
    \begin{corollary}
        For $n=4, 5, 6, 7, 9$, there are infinitely many commensurability classes of noncompact Coxeter polytopes with finite volume. Furthermore, there exists a constant $c(n)>1$ such that for all $V\gg 1$, the number of incommensurable classes of noncompact $n$-dimensional polytopes with volume less than $V$ is at least $c(n)^V$.
    \end{corollary}
    \begin{proof}
         For $n=4, 5, 6, 7, 9$, we will display some pairs of Coxeter diagrams of $n$-dimensional cusped polytopes $P_{1}^n, P_2^n,\dots$ with finite volume from \cite{mz-4d-7f}, our new list, and the list in \cite{roberts2016}. Two admissible facets $E_i^{n, \pm}$ will be marked on each $P_i^n$. A simple observation of the Coxeter diagrams and direct computer computations will show that (1) all marked facets $E_{i}^{n, \pm}$ are isometric for each $n$, and (2) $P_{1}^n, P_2^n,\dots$\footnote[2]{Not to be confused with $P_i^4, P_i^5$ in the previous two sections.} are quasi-arithmetic with the same trace field $k_n=\Q$, but have ambient groups that are pairwise not $\Q$-isogenous. Thus $P_1^n, P_2^n,\dots$ form pairs of polytopes satisfying \cref{ass:we-can-now}. The corollary below is then a direct consequence of \cref{prop:bdr-noncompact}. For the reader's convenience, we will also record the cusps, Vinberg forms, and commensurability invariants used to determine the result.
    \end{proof}
    \begin{remark}\label{rmk:issue}
        Although candidates for $n=4, 5, 6$ satisfying \cref{ass:we-can} already exist in \cite{tumarkin-nplustwo}, currently known lists of Coxeter $n$-dimensional polytopes with numbers of facets $\leq n+2$ are insufficient for one to find appropriate pairs of noncompact polytopes satisfying \cref{ass:we-can-now}. Thus, to prove the exponential growth rate of incommensurability classes via Makarov's garland in this setting, one must utilize polytopes with more facets from previously known lists, the recent list of Ma-Zheng, and our new list of five-dimensional polytopes with eight facets.
    \end{remark}
    \begin{remark}
        The exponential growth in dimensions higher than four was not achieved in \cite{bdr-incomm-4-5}, as the number of compact (quasi-)arithmetic polytopes is extremely limited.
    \end{remark}
    \subsection{Dimension four}
    The following pair of Coxeter polytopes serves as the gluing blocks for dimension four, satisfying \cref{ass:we-can-now}. They are polytopes $P_{4, 1}$ and $P_{9, 26}$ from \cite{mz-4d-7f}. The crimson-colored nodes are $E_i^{4,\pm}$. They all have type $[6, 3, 3]$. The two polytopes are both arithmetic.
    \begin{figure}[H]
        \centering
        \includegraphics[
          width=0.6\textwidth
        ]{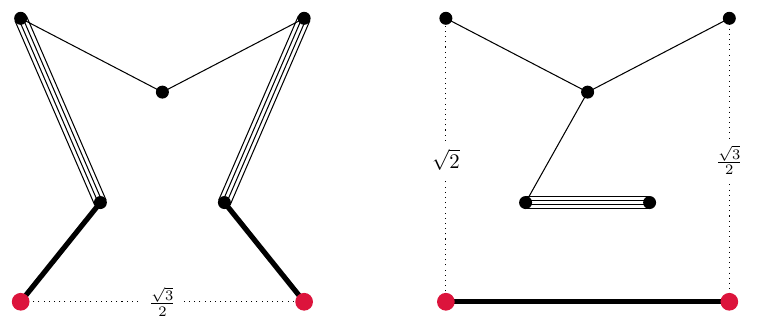}
        \caption{One pair of 4-polytopes $P_1^4$, $P_2^4$ satisfying \cref{ass:we-can-now}.}
        \label{figure:bdr4-exp}
    \end{figure}
    The polytope $P_1^4$ has two cusps
    \[F_1\cap F_3\cap F_4\cap F_5\cap F_7,\quad F_2\cap F_3\cap F_5\cap F_6\cap F_7,\]
    and the polytope $P_2^4$ has one cusp
    \[F_1\cap F_2\cap F_4\cap F_5\cap F_7.\]
    They both have finite volumes. The Vinberg forms of these polytopes are then
    \[Q_1^4=\langle 1,-2,3,1,2\rangle, \quad Q_2^4=\langle1,2,2,1,-3\rangle\]
    Despite having a determinant ratio of $1$, we note that the Hasse invariants of $Q_1^4$ and $Q_2^4$ at $p=3$ are $-1$ and $1$ resp. Thus, by \cref{thm:arith-inv}, the two polytopes are incommensurable. Observing the diagrams, they satisfy \cref{ass:we-can-now}.

    \subsection{Dimension five}
    The following Coxeter polytopes serve as the gluing blocks for dimension five, satisfying \cref{ass:we-can-now}. They are polytopes $P_{28, 1}$, $P_{28,28}$ and $P_{28, 43}$ from our new list in \cref{sec:appendix}. The crimson-colored nodes are $E_i^{5,\pm}$. The polytopes are both arithmetic. 
    \begin{figure}[H]
        \centering
        \includegraphics[
          width=0.9\textwidth
        ]{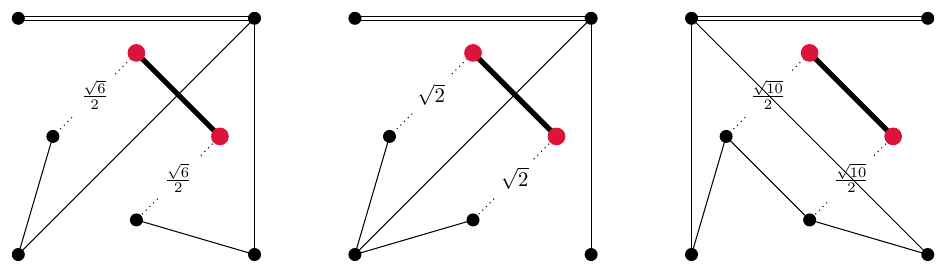}
        \caption{Five-dimensional polytopes $P_1^5$, $P_2^5$ and $P_3^5$ pairwise satisfying \cref{ass:we-can-now}.}
        \label{figure:bdr5-exp}
    \end{figure}
    The polytope $P_1^5$ has one cusp
    \[F_1\cap F_2\cap F_5\cap F_6\cap F_7\cap F_8\]
    and $P_2^5$ has two cusps
    \[F_3\cap F_4\cap F_5\cap F_6\cap F_7, \quad F_1\cap F_2\cap F_4\cap F_6\cap F_7\cap F_8,\]
    while $P_3^5$ has two cusps
    \[F_3\cap F_4\cap F_5\cap F_6\cap F_8, \quad F_1\cap F_2\cap F_5\cap F_6\cap F_7\cap F_8.\]
    They all have finite volumes. They have the following Vinberg forms, listed in the same order wrt the diagrams
    \[Q_1^5=\langle 1,-6,24,3,12,3\rangle,\quad Q_2^5=\langle 1,-8,32,8,6,16/3\rangle,\quad  Q_3^5=\langle 1,-10,40,5,20,5\rangle\]
    The Vinberg forms are pairwise not similar since the determinant ratios are all non-square. Thus, $P_1^5, P_2^5, P_3^5$ form three pairs of polytopes satisfying \cref{ass:we-can-now}.
    \begin{remark}
        In our lists, one can find many more pairs of candidates for noncompact Makarov's garlands. But it suffices to produce one pair of them for each dimension. Here we only list the most convenient ones.
    \end{remark}

    \subsection{Dimension six}
    The following two Coxeter polytopes from \cite{roberts2016} serve as the gluing blocks for dimension six, satisfying \cref{ass:we-can-now}. They are polytopes $P_1^6$ and $P_2^6$. The crimson-colored nodes are $E_i^{6,\pm}$. Note that $P_1^6$ and $P_2^6$ are both properly quasi-arithmetic.
    \begin{figure}[H]
        \centering
        \includegraphics[
          width=0.7\textwidth
        ]{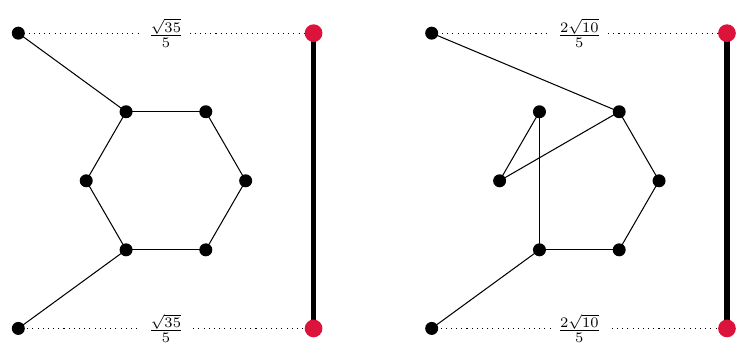}
        \caption{Two $6$-polytopes $P_1^6$ and $P_2^6$ satisfying \cref{ass:we-can-now}.}
        \label{figure:bdr6-exp}
    \end{figure}
    The polytope $P_1^6$ has one cusp 
    \[F_1\cap F_2\cap F_5\cap F_6\cap F_7\cap F_8\cap F_9\]
    and the polytope $P_2^6$ has two cusps
    \[F_3\cap F_4\cap F_5\cap F_6\cap F_8\cap F_9,\quad F_1\cap F_2\cap F_5\cap F_6\cap F_7\cap F_8\cap F_9\]
    They both have finite volumes. Their Vinberg forms are
    \[Q_1^6=\left\langle1,35,-\frac75,21,35,\frac{35}{3},\frac{35}{4}\right\rangle, \quad Q_2^6=\left\langle1,10,-\frac85,4,\frac{15}{4},\frac{10}{3},\frac52\right\rangle\]
    Despite the determinant ratio being a square, their Hasse invariants at $p=2$ are different. The ambient groups of $P_i^6$ are precisely $PO(Q_i^6)$. Thus, the polytopes have distinct ambient groups by \cref{lem:qsiarith-inv}. They satisfy \cref{ass:we-can-now}.

    \subsection{Dimension seven}
    The following two Coxeter polytopes from \cite{roberts2016} serve as the gluing blocks for dimension seven, satisfying \cref{ass:we-can-now}. They are polytopes $P_1^7$ and $P_2^7$. The crimson-colored nodes are $E_i^{7,\pm}$. Note that $P_1^7$ is properly quasi-arithmetic, while $P_2^7$ is arithmetic.
    \begin{figure}[H]
        \centering
        \includegraphics[
          width=0.7\textwidth
        ]{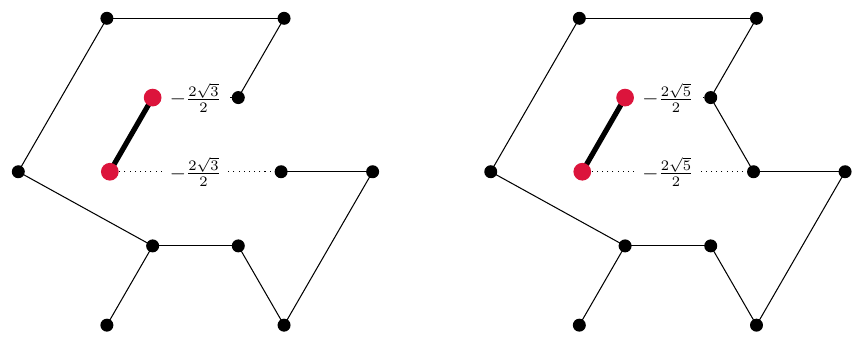}
        \caption{Two $7$-polytopes $P_1^7$ and $P_2^7$ satisfying \cref{ass:we-can-now}.}
        \label{figure:bdr7-exp}
    \end{figure}
    The polytope $P_1^7$ has two cusps 
    \[F_2\cap F_3\cap F_5\cap F_6\cap F_7\cap F_8\cap F_9\cap F_{10},\quad F_1\cap F_4\cap F_5\cap F_6\cap F_8\cap F_\cap F_{10}\]
    and the polytope $P_2^7$ has one cusp
    \[F_2\cap F_3\cap F_5\cap F_6\cap F_7\cap F_8\cap F_9\cap F_{10}\]
    They both have finite volumes. The Vinberg forms of these polytopes are
    \[Q_1^7=\left\langle 1,-1,12,\frac23,\frac58,\frac35,\frac7{12},\frac47\right\rangle, \quad Q_2^7=\left\langle 1,-\frac54,25,\frac34,\frac23,\frac58,\frac35,\frac14\right\rangle\]
    with distinct determinant square classes
    \[\det(Q_1^7)\equiv -1,\quad \det(Q_2^7)\equiv -15\pmod{(\mathbb{Q}^{\times})^2}.\]
    The ambient groups of $P_i^7$ are precisely $PO(Q_i^7)$. Thus, the polytopes have distinct ambient groups by \cref{lem:qsiarith-inv}. They satisfy \cref{ass:we-can-now}.

    \subsection{Dimension nine}
    The following two Coxeter polytopes from \cite{roberts2016} serve as the gluing blocks for dimension nine, satisfying \cref{ass:we-can-now}. They are polytopes $P_1^9$ and $P_2^9$. The crimson-colored nodes are $E_i^{9,\pm}$. Both polytopes are arithmetic.
    \begin{figure}[H]
        \centering
        \includegraphics[
          width=0.7\textwidth
        ]{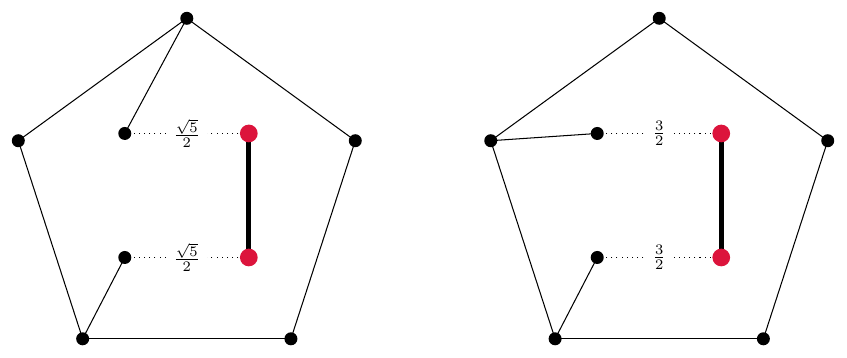}
        \caption{Two nine-dimensional polytopes $P_1^9$ and $P_2^9$ satisfying \cref{ass:we-can-now}.}
        \label{figure:bdr9-exp}
    \end{figure}
    The polytope $P_1^9$ has one cusp
    \[F_2\cap F_3\cap F_5\cap F_6\cap F_7\cap F_8\cap F_9\cap F_{10}\cap F_{11}\cap F_{12}\]
    and the polytope $P_2^9$ has four cusps
    \[\begin{aligned}
        v_1&=F_1\cap F_4\cap F_5\cap F_6\cap F_7\cap F_8\cap F_9\cap F_{10}\cap F_{11},\\
        v_2&=F_1\cap F_4\cap F_5\cap F_6\cap F_7\cap F_8\cap F_{10}\cap F_{11}\cap F_{12},\\
        v_3&=F_1\cap F_4\cap F_6\cap F_7\cap F_8\cap F_9\cap F_{10}\cap F_{11}\cap F_{12},\\
        v_4&=F_2\cap F_3\cap F_5\cap F_6\cap F_7\cap F_8\cap F_9\cap F_{10}\cap F_{11}\cap F_{12}.
    \end{aligned}\]
    Their Vinberg forms are
    \[Q_1^9=\left\langle5,-\frac14,5,\frac{15}{4},\frac{10}{3}, \frac{25}{8},3,\frac{35}{12},\frac{20}{7},\frac54\right\rangle,\quad Q_2^9=\langle 5,-1,5,15,30,2,3,105,35,5\rangle\]
    with distinct determinant square classes
    \[\det(Q_1^9)\equiv -5,\quad \det(Q_2^9)\equiv -1\pmod{(\mathbb{Q}^{\times})^2}.\]
    The ambient groups of $P_i^9$ are precisely $PO(Q_i^9)$. Thus, the polytopes have distinct ambient groups by \cref{lem:qsiarith-inv}. Hence, we have a pair satisfying \cref{ass:we-can-now}.
    
    \newpage
    \section*{Appendix: Coxeter diagrams of five-dimensional hyperbolic Coxeter polytopes with eight facets}\label{sec:appendix}
    We display all Coxeter diagrams of finite-volume five-dimensional polytopes with eight facets. A dotted edge indicates an ultraparallel facet relation, while a bold edge indicates an asymptotic facet relation.

    \begin{figure}[htbp]
        \centering
        \includegraphics[
          width=1\textwidth,
          clip
        ]{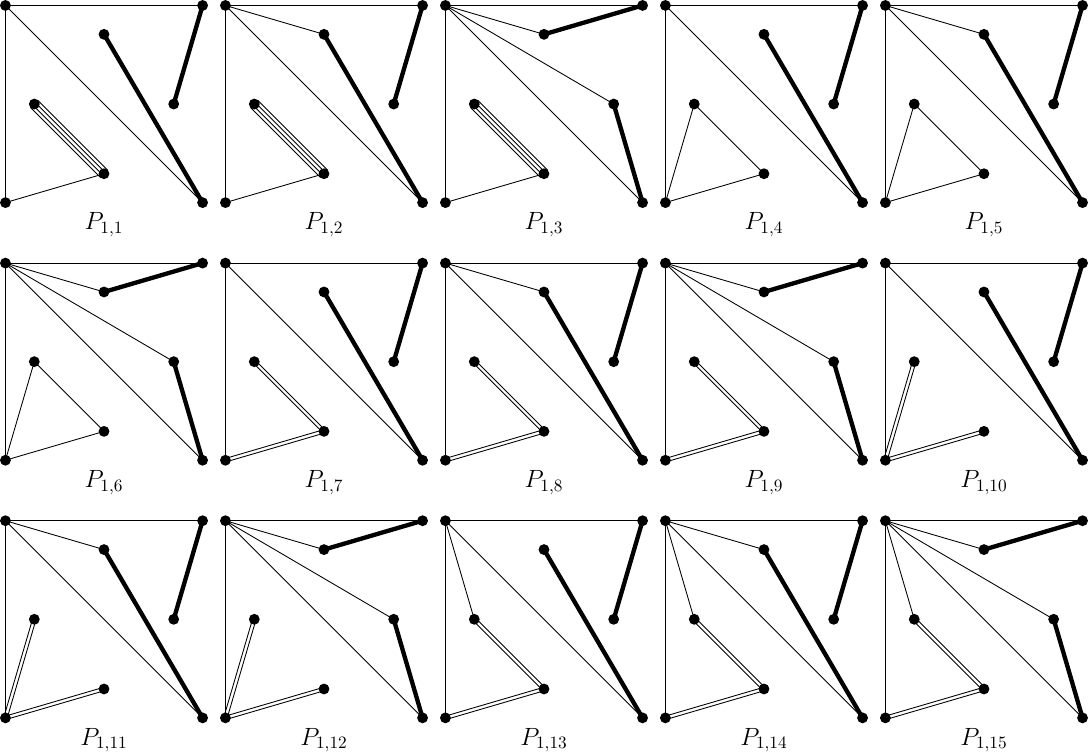}
        \caption{Finite-volume Coxeter $5$-polytopes with $8$ facets over polytopes $P_1$.}
        \label{figure:p50}
    \end{figure}
    
    \begin{figure}[htbp]
        \centering
        \includegraphics[
          width=0.9\textwidth,
          clip
        ]{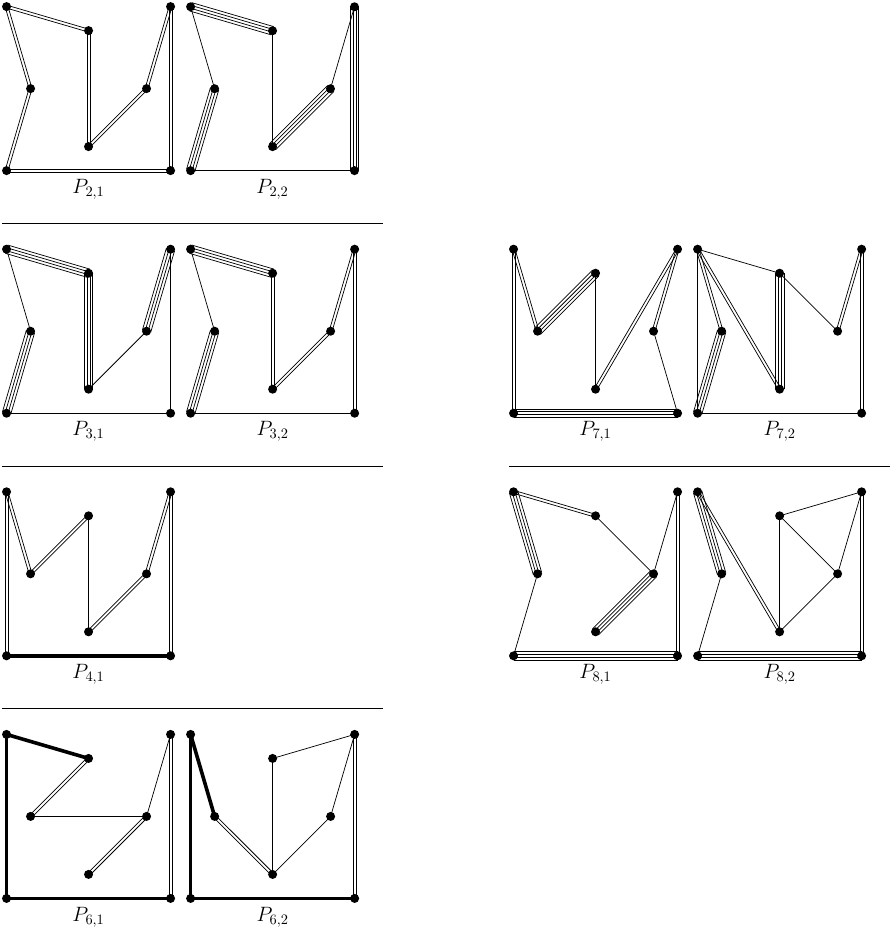}
        \caption{Finite-volume Coxeter $5$-polytopes with $8$ facets over polytopes $P_i$ for $i=2, 3, 4$, $6, 7, 8$, respectively.}
        \label{figure:p51}
    \end{figure}

    \begin{figure}[htbp]
        \centering
        \includegraphics[
          width=1\textwidth,
          clip
        ]{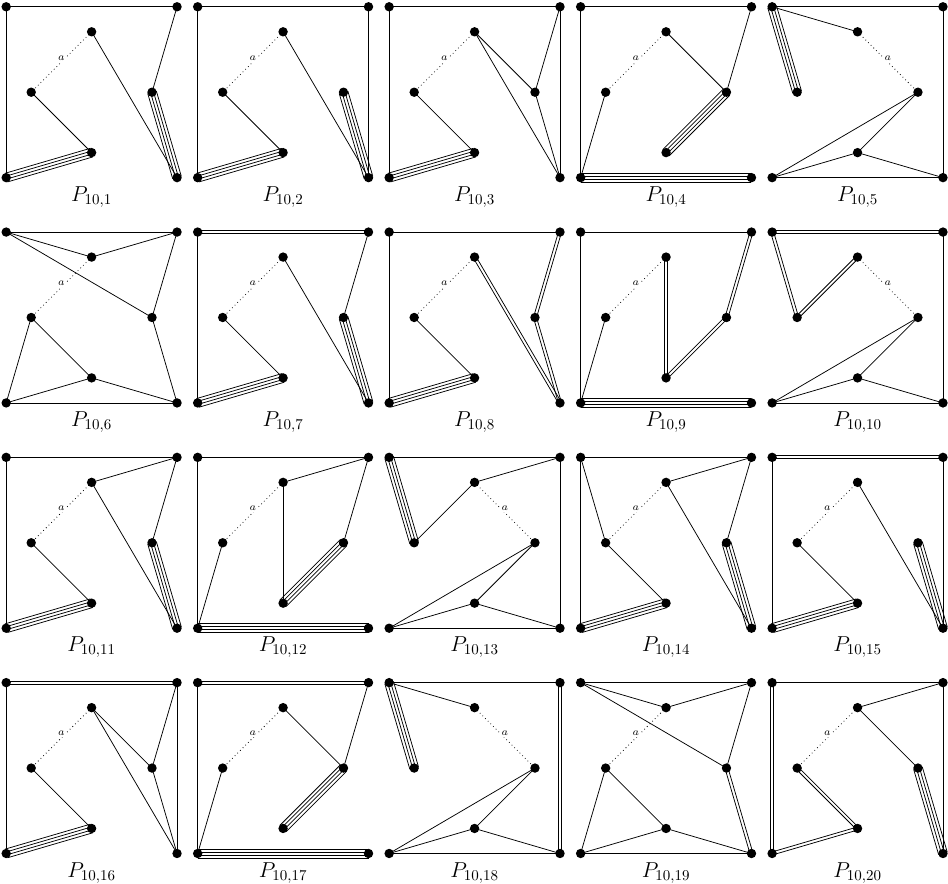}
        \caption{Finite-volume Coxeter $5$-polytopes with $8$ facets over polytopes $P_10$.} \label{figure:p52}
    \end{figure}
    
    \begin{figure}[htbp]
        \includegraphics[
          width=1\textwidth,
          clip
        ]{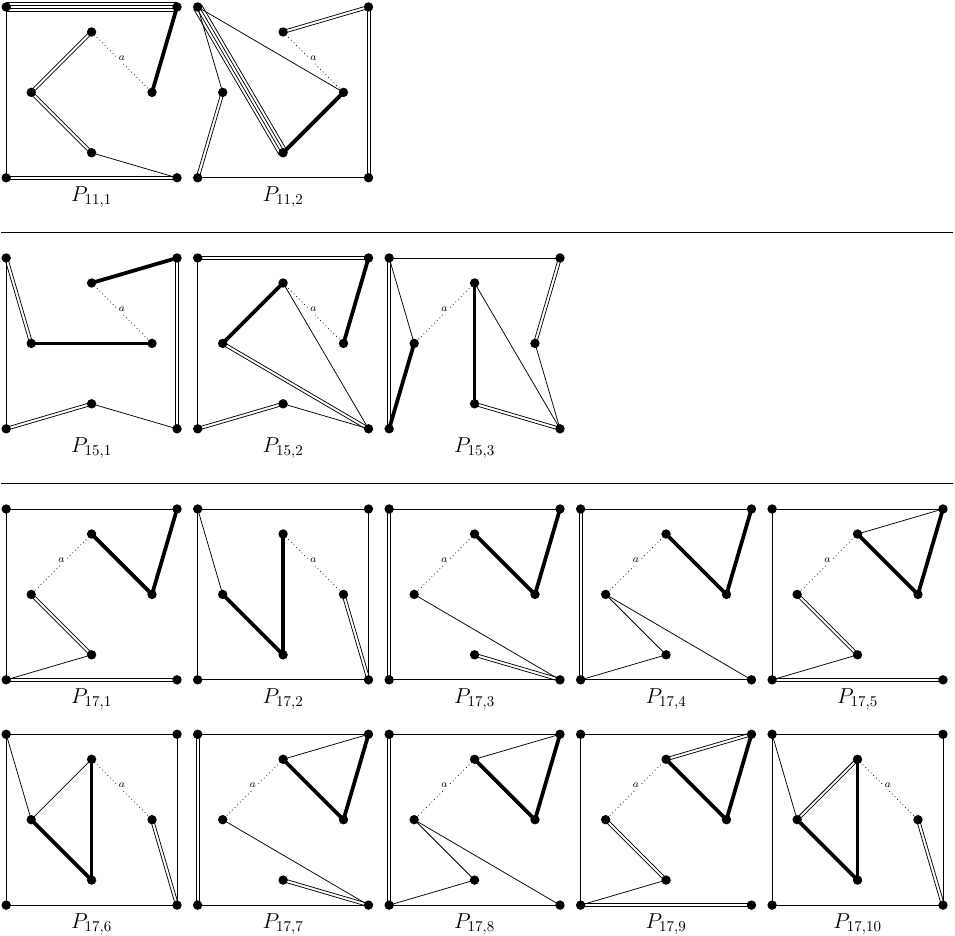}
        \caption{Finite-volume Coxeter $5$-polytopes with $8$ facets over polytopes $P_i$ for $i=11, 15, 17$, respectively.}
        \label{figure:p53}
    \end{figure}

    \begin{figure}[htbp]
        \includegraphics[
          width=1\textwidth,
          clip
        ]{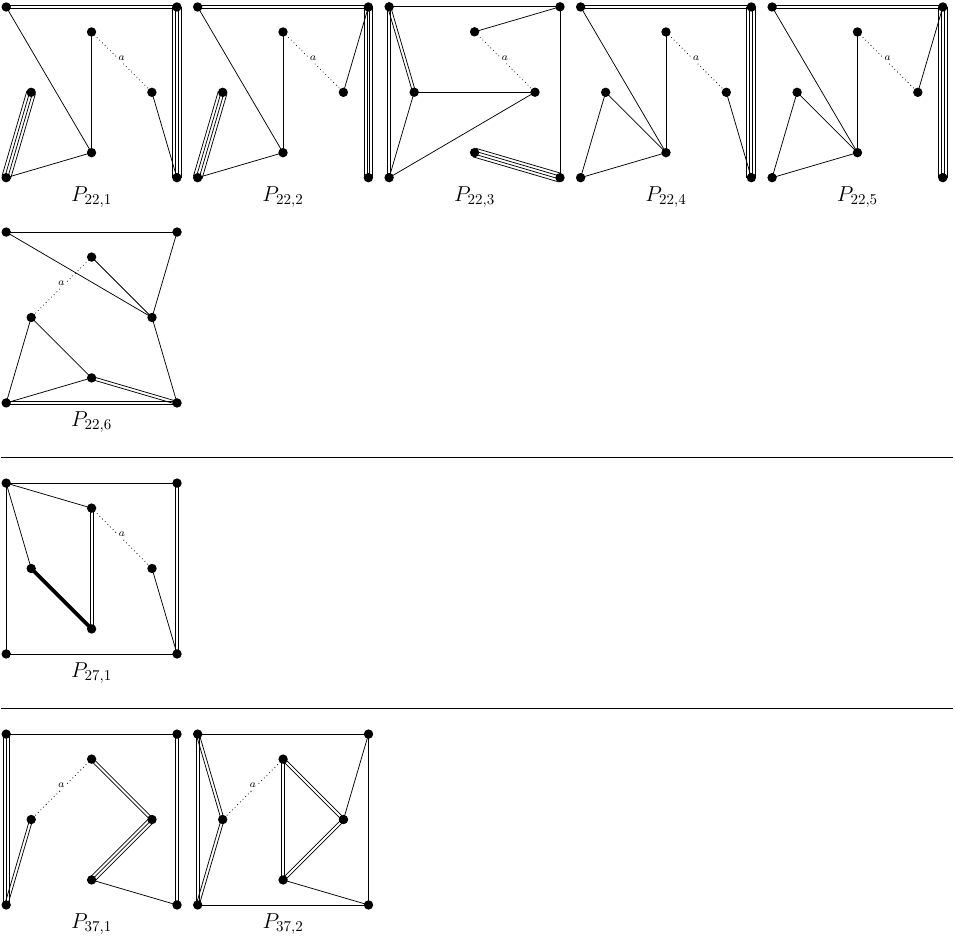}
        \caption{Finite-volume Coxeter $5$-polytopes with $8$ facets over polytopes $P_i$ for $i=22, 27, 37$, respectively.} \label{figure:p54}
    \end{figure}
  
    \begin{figure}[htbp]
        \includegraphics[
          width=0.93\textwidth,
          clip
        ]{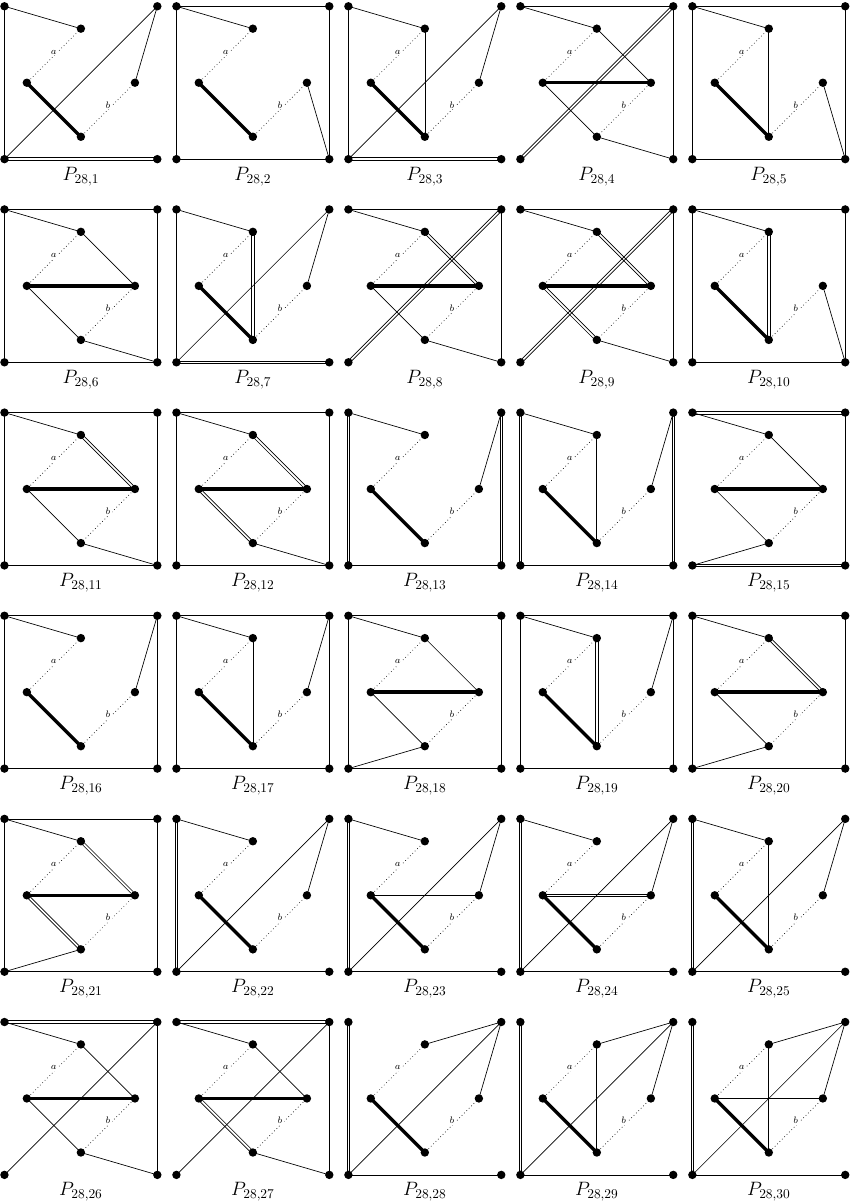}
        \caption{Finite-volume Coxeter $5$-polytopes with $8$ facets over polytopes $P_{28}$, part one.} \label{figure:p55}
    \end{figure}
    
    \begin{figure}[htbp]
        \includegraphics[
          width=1\textwidth,
          clip
        ]{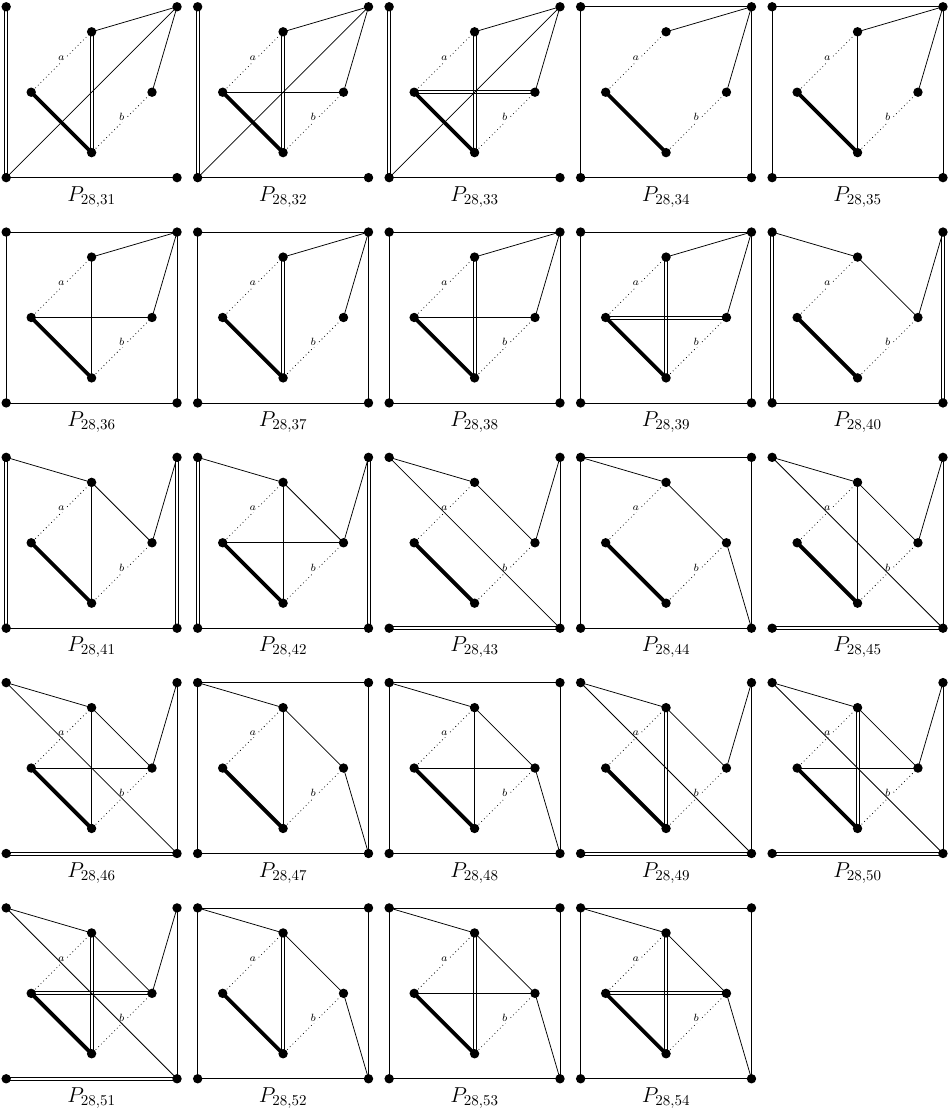}
        \caption{Finite-volume Coxeter $5$-polytopes with $8$ facets over polytopes $P_{28}$, part two.} \label{figure:p56}
    \end{figure}

    \begin{figure}[htbp]
        \includegraphics[
          width=1\textwidth,
          clip
        ]{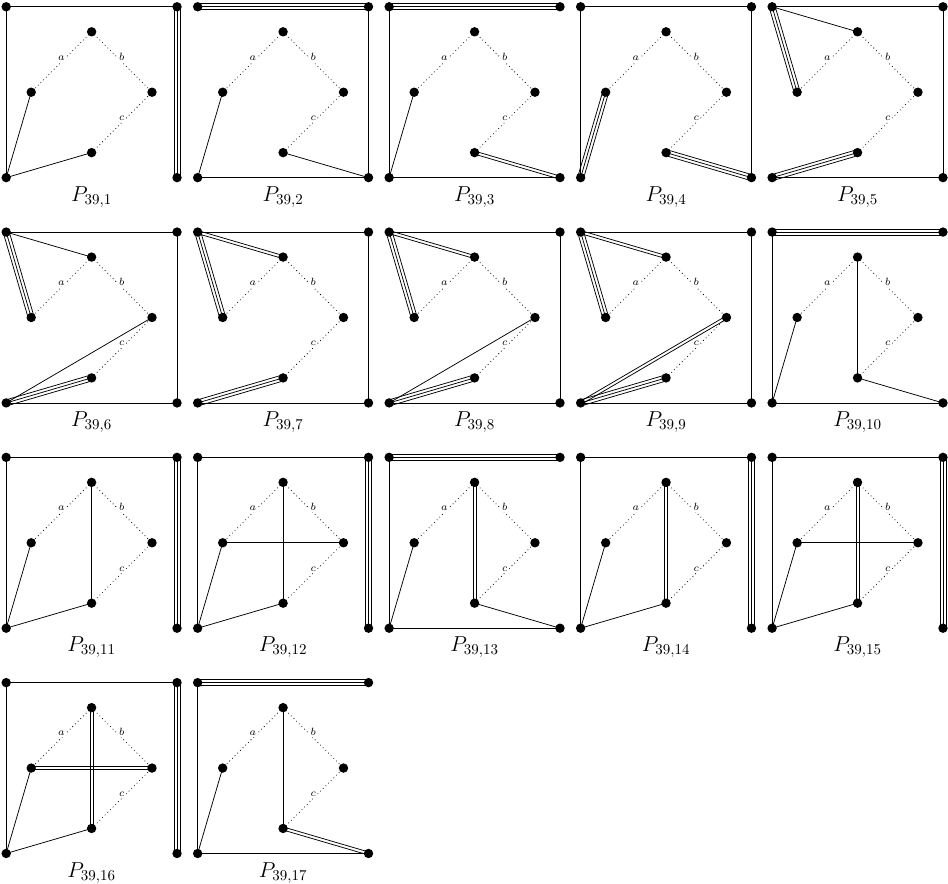}
        \caption{Finite-volume Coxeter $5$-polytopes with $8$ facets over polytopes $P_{39}$.} \label{figure:p57}
    \end{figure}

    \clearpage
    \printbibliography
\end{document}